\documentclass[12pt,reqno]{amsart}
\usepackage[english]{babel}
\usepackage{mathrsfs, csquotes}
\usepackage{amsfonts,amssymb,amsmath,mathtools,commath,braket,amssymb}
\usepackage{dsfont} 

\usepackage{color} \definecolor{bleu_sombre}{rgb}{0,0,0.6}  \definecolor{rouge_sombre}{rgb}{0.8,0,0}\definecolor{vert_sombre}{rgb}{0,0.6,0}
\usepackage[plainpages=false,colorlinks,linkcolor=bleu_sombre,
citecolor=rouge_sombre,urlcolor=vert_sombre,breaklinks]{hyperref}
\usepackage{enumerate}
\usepackage{enumitem}  
\usepackage{stmaryrd}  

  \usepackage[all]{xy}

\usepackage{appendix}
\theoremstyle{plain}
\newtheorem{theorem}{Theorem}[section]
\newtheorem{lemma}[theorem]{Lemma}
\newtheorem{corollary}[theorem]{Corollary}
\newtheorem{proposition}[theorem]{Proposition}

\theoremstyle{definition}
\newtheorem{remark}[theorem]{Remark}
\newtheorem{notation}{Notation}

\newtheorem{definition}[theorem]{Definition}
\newtheorem{assumption}[theorem]{Assumption}

\numberwithin{equation}{section}

  \def\cC{\mathcal{C}}
  \def\cF{\mathcal{F}}
  \def\cI{\mathcal{I}}
  \def\cL{\mathcal{L}}
\def\cM{\mathcal{M}} \def\cN{\mathcal{N}} \def\cO{\mathcal{O}}
\def\cP{\mathcal{P}} \def\cQ{\mathcal{Q}}

\def\NN{\mathbb{N}}
\def\ZZ{\mathbb{Z}}

\def\RR{\mathbb{R}}
\def\CC{\mathbb{C}}

\def\({\left(}
\def\){\right)}

\def\<{\left\langle}
\def\>{\right\rangle}

\def\leq{\leqslant}
\def\geq{\geqslant}

\def\lc{\left[}
\def\rc{\right]}

\def\lb{\left\{}
\def\rb{\right\}}

\def\llc{\llbracket}
\def\rrc{\rrbracket}

\def\ie{\textrm{~i.e.~}}

\newcommand{\valabs}[1]{\left\lvert #1 \right\rvert}  

\newcommand{\autopar}[1]{\left( #1 \right)} 

\newcommand{\diff}{\mathop{}\mathopen{}\mathrm{d}}  

\newcommand{\Id}{\operatorname{Id}}

\newcommand{\be}{\begin{equation}}
\newcommand{\ee}{\end{equation}}

\newcommand{\bea}{\begin{eqnarray}}
\newcommand{\eea}{\end{eqnarray}}

\newcommand{\bee}{\begin{eqnarray*}}
\newcommand{\eee}{\end{eqnarray*}}

\begin{document}

\title[Semiclassical spectrum of the Dirichlet-Pauli operator on an annulus]{Semiclassical spectrum \\ of the Dirichlet-Pauli operator \\ on an annulus}
\author[E. Lavigne Bon]{E. Lavigne Bon}
\email{enguerrand.lavigne-bon@univ-amu.fr}
\address[E. Lavigne Bon]{Aix-Marseille Univ., CNRS, I2M}

    \begin{abstract}
This paper is devoted to the semiclassical analysis of the spectrum of the Dirichlet-Pauli operator on an annulus. We assume that the magnetic field is strictly positive and radial. We give an explicit asymptotic expansion at the first order of the lowest eigenvalues of this operator in the semiclassical limit. In particular, we exhibit the Aharonov-Bohm effect that has been revealed, for constant magnetic field, in a recent paper by Helffer and Sundqvist.
    \end{abstract}

\maketitle

\vspace{-0.5cm}

\begin{footnotesize}
\tableofcontents
\end{footnotesize}

\newpage

            \section{Introduction} \label{section.introduction}

This paper deals with the low-lying spectrum of the Dirichlet-Pauli operator. This operator is the Hamiltonian of the Pauli's equation, which models non-relativistic spin-1/2 particules interacting with a magnetic field. For a particle of mass and electric charge equal to 1, in a pure magnetic field perpendicular to the domain, the Pauli's equation reads
    \begin{equation*}
    ih \partial_t \Psi = \cP_h \, \Psi, 
    \end{equation*}
where $\Psi : \RR \times \Omega \to \CC^2$ is the state of the system, $\Omega$ a smooth bounded domain of $\RR^2$ and $h>0$ the semiclassical parameter. We associate to the magnetic field (which is assumed to be smooth), denoted by $B$, a magnetic vector potential $A= (A_1 , A_2 )$ satisfying $B = \partial_x A_2 - \partial_y A_1$. \\
The Dirichlet-Pauli operator acts as
    \begin{equation*}
    \cP_h = \left[  \sigma  \cdot \left( -ih\nabla - A \right) \right]^2 \; \; \text{on} \; \; \, \mathrm{H}^2 \left( \Omega,\CC^2 \right)  \cap \mathrm{H}^{1}_{0} \left( \Omega,\CC^2 \right),
    \end{equation*}
with $\sigma = \left( \sigma_1 , \sigma_2 \right)$, the Pauli matrices (Hermitian, unitary)
    \begin{equation} \label{eq.matrices-Pauli}
\sigma_1 =
\begin{pmatrix}
0 & 1 \\
1 & 0
\end{pmatrix},
\quad
\sigma_2 =
\begin{pmatrix}
0 & -i \\
i & 0
\end{pmatrix},
\quad
\sigma_3 =
\begin{pmatrix}
1 & 0 \\
0 & -1
\end{pmatrix}  ,
    \end{equation}
and $\sigma \cdot \mathrm{\mathbf{v}} = \sigma_1  \mathrm{v_1} + \sigma_2  \mathrm{v_2} $ for $\mathrm{\mathbf{v}} = \left(  \mathrm{v_1} ,  \mathrm{v_2} \right) \in \CC^2$. Notice that to ensure its selfadjointness we assume that boundary carries Dirichlet conditions.
Let $\left(  \lambda_k (h) \right)_{k \in \NN^{*}}$ be the real non-decreasing sequence of eigenvalues of $\cP_h$ counting multiplicity.

\medskip

The Dirichlet-Pauli operator has been the subject of many recent works. In particular, the paper \cite{barbaroux:hal-01889492} presents a new approach, for simply connected domains, leading to very accurate estimates of the low-lying spectrum. Helffer and Sundqvist \cite{helffer2017semi} proved the exponential decay rate of the ground state in the case of a connected domain. They also given, in the same paper, numerical results describing the behavior of the first eigenvalues on an annulus with a constant magnetic field.

Here, we consider an annulus $\Omega$ and we assume that the magnetic field is strictly positive and radial.
We give explicitly the first term of the asymptotic expansion for the first eigenvalues. Let us informally describe our main result. We prove in this paper that, for all $k\in \NN^*$,
\begin{equation*} 
    \lambda_k (h) =  \min_{\tiny\begin{array}{c}V\subset\mathbb{Z}\\\# V = k\end{array}}\max_{m\in V} \, f \left( m- \frac{c_0}{h} \right) \, \sqrt{h} \; e^{\frac{2 \phi_{\min}}{h}} \; (1+o_{h \rightarrow 0}(1)) \, ,
\end{equation*}
with $c_0 \in \mathbb{R}$, $\phi_{\min}$ a strictly negative constant and $f :  \mathbb{R} \rightarrow  \mathbb{R}$ a coercive function known explicitly according to the magnetic field and $\Omega$.

The prefactor in $\lambda_k (h)$ encodes two kinds of gauge invariance. Of course, the standard magnetic gauge invariance (adding a gradient to the magnetic
potential) leaves $\phi_{\rm min}, f$ and $c_0$ invariant. The presence
of the hole introduces another degree of freedom given by conjugating the operator by $e^{ip\, \mathrm{arg}}$ ($p\in\mathbb{Z}$, $\mathrm{arg}$ being the argument). The invariance by translation by an integer of the prefactor of $\lambda_k(h)$ is reminiscent of this invariance (see Section \ref{subsection.class.potentiel}). 

Note that the constant $c_0$ depends mainly on the circulation of the magnetic potential on the interior boundary. When it does not vanish, we observe an oscillation of the eigenvalues in the semiclassical limit through the 1-periodicity of 
\begin{equation*} 
    \displaystyle d \longmapsto  \min_{\tiny\begin{array}{c}V\subset\mathbb{Z}\\\# V = k\end{array}}\max_{m\in V} \, f \left( m- d \right).
\end{equation*}
This influence of the circulation of the magnetic potential on the eigenvalues evokes the Aharonov-Bohm effect. It has already been mentioned in \cite{helffer2017semi}, for more details on this phenomenon see the original reference  \cite{aharonov1959significance}, for a new point of view see also \cite{de2020new}.\\
This result highlights some phenomena that should also appear in doubly connected case (without symetry) and in multiconnected domains.
        \subsection{Hypothesis and definition of the Dirichlet-Pauli operator} \label{subsection.hypo_and_Pauli_operator} ~ \\
Let $\Omega \subset \mathbb{R}^2$ be an annulus centered at the origin with radius $0 < \rho_1 < \rho_2 $. We let
\begin{equation} \label{eq.Omega}
    \partial \Omega = \partial\Omega_{\mathrm{int}} \bigsqcup \partial\Omega_{\mathrm{ext}},
\end{equation}
with $\partial\Omega_{\mathrm{int}} = \cC (0, \rho_1 )$ and $\partial\Omega_{\mathrm{ext}} = \cC (0, \rho_2 )$.\\
Consider  a magnetic field $B \in \mathcal{C}^{\infty} \left( \overline{\Omega} , \RR \right) $. Despite the presence of a hole, there exists a regular vector potential denoted by $A= ( A_1 , A_2 )$, which satisfies
    \begin{equation} 
    \label{eq.rotationnel}
    B= \mathrm{rot} \autopar{A} = \partial_{x} A_2 - \partial_{y} A_1  .
    \end{equation}
Note however that two Pauli operators with two vector potentials associated with $B$ are not necessarily unitarily equivalent (see Section \ref{section.invariants}).
    \begin{assumption} \label{hypothese1}
    The magnetic field is {\bf radial} and {\bf strictly positive}.
    \end{assumption}
    \subsubsection{The Dirichlet-Pauli operator}
We are interested in the Dirichlet-Pauli operator $\autopar{\cP_h , \mathrm{Dom}\autopar{\cP_h}}$ defined for all $h>0$ as
$$
\cP_h = \lc \sigma  \cdot \autopar{{\bf p} - A} \rc^2 =  
\begin{pmatrix}
\vert {\bf p} - A \vert^2 -hB & 0 \\ 
0 & \vert {\bf p} - A \vert^2 +hB
\end{pmatrix}=\begin{pmatrix}
\cL^{-}_h & 0 \\
0 & \cL^{+}_h
\end{pmatrix} ,
$$
acting on the domain
$$
\mathrm{Dom}\autopar{\cP_h} =  \mathrm{H}^2 \left( \Omega,\CC^2 \right)  \cap \mathrm{H}^{1}_{0} \left( \Omega,\CC^2 \right).
$$
Here $ {\bf p} = -ih \nabla $.
This operator is non-negative, self-adjoint with compact resolvent. By the spectral theorem, the spectrum of $\autopar{\cP_h , \mathrm{Dom}\autopar{\cP_h}}$ is real, discrete and can be written as a sequence tending to $+ \infty$. 
\\
The purpose of this paper is to investigate the behavior of the lowest eigenvalues in the semiclassical limit. Since $B>0$ on $\overline{\Omega}$, it is enough to study the spectrum of $\cL^{-}_h$.
    \begin{notation} \label{notation.valeurs_propres}
    Let $\autopar{\lambda_k (h)}_{k\in\NN^{*}} $ denote the non-decreasing sequence of eigenvalues of $\mathcal{P}_h$. 
    \end{notation}
By the min-max theorem, we have the following characterization
    \begin{equation}
    \label{notation.minmax1}
    \lambda_k (h) = \underset{\underset{\mathrm{dim}V = k}{V \subset H_{0}^{1} ( \Omega , \CC^{2} ) }}{\mathrm{inf}} \; \underset{u\in V\backslash \{ 0 \} }{\mathrm{sup}} \; \frac{\left\|\sigma \cdot \autopar{{\bf p } - A} u \right\|_{L^{2}\autopar{\Omega}}^2}{\left\|u\right\|_{L^{2}\autopar{\Omega}}^2} .
    \end{equation}
    \subsubsection{Scalar potential} \label{subsubsection.phi}
The choice of $A$ will play an important role. A particular choice is associated with the scalar potential, $\phi$ being the unique solution in $H_0^1 \left( \Omega , \RR \right)$ (cf. \cite[Theorem 6, p.326]{evans2010partial}) of the Poisson equation
    \begin{equation}
    \label{eq.phi}
      \left\{
      \begin{array}{ll}
        \Delta \phi = B & \text{on} \; \Omega \\
        \phi = 0  & \text{in} \; \partial \Omega .
     \end{array}
    \right.
    \end{equation}
Since $B$ is positive, $\phi$ is subharmonic and satisfies
$$
\underset{x \in \overline{\Omega}}{\mathrm{max}}\;  \phi = \underset{x \in \partial\Omega}{\mathrm{max}} \; \phi = 0 \;,
$$
by the maximum principle.
In particular, the minimum of $\phi$ is negative and attained in $\Omega$. We note that the exterior normal derivative of $\phi$, denoted $\partial_n \phi$, is strictly positive on $\partial \Omega$ from \cite[Hopf's Lemma, p.330]{evans2010partial}. 
    \begin{remark} \label{remark.phimin}
Assumption  \ref{hypothese1} on the magnetic field and the uniqueness of $\phi$ ensures that $\phi$ is radial and admits a unique circle of minimum centered at the origin and of radius $\mathrm{r}_{\mathrm{min}} \in ]\rho_1 , \rho_2[$. We note $\phi_{\min}$ the minimum of $\phi$. In polar coordinates, $\phi$ is the solution of 
\begin{equation*}
  \left\{
  \begin{array}{ll}
     \phi''(r) + \frac{1}{r} \phi' (r) = B(r) & \text{on } \; ]\rho_1 , \rho_2[ \\
    \phi\autopar{\rho_1} = \phi \autopar{\rho_2} =0 \; . &
 \end{array}
\right.
\end{equation*}
so that $\phi '' (r_{\min})=B(r_{\min}) \geq B_0=\inf\{B(x) : x\in\Omega\}$.  
    \end{remark}

    \subsection{Results and discussions} \label{subsection.results_discus}

Some recent works have investigated the semiclassical limit of the bottom of the spectrum:

\medskip

    \begin{enumerate}
    \item \label{results_discus.pt1} In the non-simply connected case, Helffer and Sundqvist have proved, in \cite{helffer2017semi}, that if the magnetic field is positive, then
$$
 \lambda_1 (h)  = \exp{\left( \frac{2 \phi_{\min}}{h} + o_{h\rightarrow 0} \left( \frac{1}{h} \right)\right)}.
$$
The techniques employed for the lower bound are mainly based on those used by Ekholm, Kova{\v{r}}{\'\i}k and Portmann in \cite{ekholm2016estimates}. The main novelty of their proof is the combination of some gauge invariances and the Hodge-de Rham theory to control the oscillations induced by the circulation of the magnetic potential.\\
In \cite[Section 7]{helffer2017semi}, a numerical analysis of the smallest eigenvalues is realized in the case of a constant magnetic field  on the annulus, by means of a finite difference method. It relies on a Fourier decomposition, parametrized by the circulation of the magnetic potential. 
    \item \label{results_discus.pt2} In the simply connected case, Barbaroux, Le Treust, Raymond and Stockmeyer have proved, in \cite{barbaroux:hal-01889492}, that for all $k\in \NN^*$, there exist $C_{inf}(k), C_{sup}(k) >0$ such that
$$
C_{inf}(k) h^{1-k} e^{2 \phi_{\min} / h}(1 + o_{h \rightarrow 0}(1)) \leq \lambda_k (h) \leq C_{sup}(k) h^{1-k} e^{2 \phi_{\min} / h}(1 + o_{h \rightarrow 0}(1)) \,,
$$
under the assumption that $B$ is positive and that $\phi$ has a unique minimum, which is non-degenerate.
The prefactors $C_{inf} (k)$ and $C_{sup} (k)$, are given explicitly. Their strategy is based on the Riemann mapping theorem and the connection between the spectral analysis of Dirichlet-Pauli operator and Cauchy-Riemann operators. In particular, the lower bound, established by a holomorphic approximation result \cite[Proposition 5.4]{barbaroux:hal-01889492}, is a consequence of the ellipticity of Cauchy-Riemann operators.
    \end{enumerate}
\medskip
The main result of this article is the following.
    \begin{theorem} \label{theorem.principal} 
Let $B\in \cC^\infty \autopar{\overline{\Omega}, \RR}$ be radial such that 
$$B_0 = \mathrm{inf} \{  B(x), \; x \in \Omega  \} > 0,$$ 
and $A\in \mathcal{C}^{\infty} \left( \overline{\Omega},\RR^2 \right)$ be an associated vector potential. \\
Then, for all fixed $k\in \NN^*$, we have
    \begin{equation*}
    \lambda_{k} (h) = \alpha_k (h) \; \sqrt{h} \; e^{2\phi_{\min} / h} (1 + o_{h\to 0}(1)),
    \end{equation*}
where
$$
\alpha_k(h) = \min_{\tiny\begin{array}{c}V\subset\ZZ\\\# V = k\end{array}}\max_{m\in V} \, f \autopar{m- \frac{c_0}{h} }  \; \text{and} \; \; c_0 =  \rho_1 \partial_r \phi \autopar{\rho_1 } -  \frac{1}{2\pi}\int_{\partial \Omega_{int}} A\,,
$$
with $f :  \RR \rightarrow  \RR$ given by
$$
f(m) =  2 \sqrt{\frac{B(r_{\min})}{\pi}} \autopar{\partial_n\phi\autopar{\rho_1} \autopar{\frac{\rho_1}{r_{\min}} }^{2m+1} + \partial_n\phi\autopar{\rho_2} \autopar{\frac{\rho_2}{r_{\min}} }^{2m+1}  } \, .
$$
    \end{theorem}
    \begin{remark} We note that, unlike \cite[Theorem 1.3]{barbaroux:hal-01889492}, the decreasing behavior in $h$ is the same for each of the eigenvalues. This comes from the fact that the scalar potential does not admit a unique non-degenerate minimum.
    \end{remark}
We give here a brief outline of the ideas to establish the main theorem.
    \begin{enumerate}[label = \roman*.]
    \item In Section \ref{section.invariants}, we follow \cite{helffer2017semi} to select a useful gauge. The explicit description of the Hodge-de Rham theory reveals the role of circulations when writing the magnetic potential, see Proposition \ref{propo.potentiel.flux}. In Section \ref{subsection.fibration_pauli}, we make a fibration of the Dirichlet-Pauli operator by means of Fourier series: 
$$
\cL^{-}_{h} = \bigoplus_{m\in \ZZ} \, \cL_{h,m}^- .
$$

    \item  Techniques used in \cite{barbaroux:hal-01889492} and in Sections \ref{section.estim.elliptiques} and \ref{sec:majoration} share common features. First, we prove uniform ellipticity inequalities for the Dirac operator for fixed $m$. Then, we deduce a lower bound for $k \geq 2$ of the $k$-th eigenvalue $\lambda_{k,m} (h)$ of $\cL_{h,m}^-$. Finally, with the upper bound on $\lambda_{1,m} (h)$, we deduce a localization and a monomial approximation for the eigenfunctions associated to $\lambda_{1,m} (h)$. 
    \item  The proof of Theorem \ref{theorem.principal} relies on Proposition \ref{propo.CVU.coercivite}:
    \begin{enumerate}
    \item The lowest eigenvalues of $\cL^-_h$ are, in the semiclassical limit, ground states of $\cL^-_{h,m}$,  see Lemma \ref{lemme:j=1}.
    \item The sequence of renormalized eigenvalues,
$$
f_{1,h} (m) = \frac{\lambda_{1,m} (h)}{ \sqrt{h} \, e^{2 \phi_{\min} /h}} ,
$$
with $\phi_{\min}$ defined in Remark \ref{remark.phimin}, converges uniformly (in $m$) on any compact of $\RR$, when $h$ tends to $0$. Moreover, it also verifies a property of weak coercivity with respect to $m$.
    \end{enumerate}
Finally, if $m(h)$ is such that $\lambda_{1,m(h)} (h) = \lambda_k (h) $, the upper bound of Lemma \ref{lem.deflimf1h} implies that $m(h)$ is uniformly bounded with respect to $h$. The min-max formula on integers with the uniform convergence on any compact ends the proof of Theorem \ref{theorem.principal}.
    \end{enumerate}
            \section{Choice of gauges on the annulus} \label{section.invariants}
When the domain is simply connected the vector potential $A$ can be chosen, via gauge invariance, equal to $\nabla^{\perp} \phi $, where $\phi$ is a solution of \eqref{eq.phi}, modulo the gradient of a regular function \cite{barbaroux:hal-01889492}.\\
In this section we slightly revisit \cite{helffer2017semi} by describing explicitly an equivalence class of magnetic potentials defined in such a way that the set of associated Pauli operators are unitarily equivalent.
        \subsection{Selection of vector potential} \label{subsection.selec_potentiel}
Recall that $A$ is fixed and satisfies \eqref{eq.rotationnel}.
By gauge invariance, we can choose a new magnetic potential $\mathbf{ \check{A} }$ such that
    \begin{equation}
    \label{eq.potentiel.rot.div.norm}
      \left\{
          \begin{array}{ll}
            \mathrm{rot}(\mathbf{ \check{A} }) = B \; \;  \text{and} \; \; \mathrm{div}({\bf \check{A}})=0 & \text{on} \; \Omega, \\
            {\bf \check{A}}\cdot \mathbf{n} =0 & \text{in} \; \partial\Omega.
        \end{array}
        \right.
    \end{equation}
In fact, consider $f$ a solution of
$$
  \left\{
      \begin{array}{ll}
        -\Delta f = \mathrm{div}(A ) & \text{on} \; \Omega \\
        \nabla f \cdot \mathbf{n} = - A \cdot \mathbf{n} & \text{in} \; \partial \Omega,
     \end{array}
    \right.
$$
Such a solution exists but is not unique, see \cite[Theorem 5.2.18]{allaire2005analyse}.\\
We can easily check that $ \mathbf{ \check{A} } = A + \nabla f $ verify \eqref{eq.potentiel.rot.div.norm}. \\
The unitary transformation to be carried out on the operator to make such a modification is the following
$$
\sigma  \cdot \autopar{{\bf p} - \mathbf{ \check{A} } } = \exp \left( i \frac{f}{h} \Id \right) \; \sigma  \cdot \autopar{{\bf p} - A} \; \exp \left( - i \frac{f}{h} \Id \right) .
$$
        \subsection{A class of admissible vectors potential} \label{subsection.class.potentiel}
Let us describe a family of vector potentials associated with $B$. To do so, we will need the following two lemmas, consequences of the Hodge de-Rham theory, whose proofs are provided in Appendix \ref{annexe1}.

    \begin{lemma} \label{lemme.fct.nulle}
Let $F \in \cC^\infty \left( \Omega, \RR^2 \right)$ a vector potential satisfying \eqref{eq.potentiel.rot.div.norm} with $B=0$ and
$$
\int_{\partial \Omega_{int}} F = 0,
$$
where $\partial \Omega_{int}$ is defined in \eqref{eq.Omega}.\\
Then, $F=0$.
    \end{lemma}
    \begin{lemma} \label{lemme.theta}
Let $\theta$ the unique solution of
    \begin{equation*}
      \left\{
      \begin{array}{lllll}
        \Delta \theta = 0 & \text{on} \; \Omega & & & \\
        \theta = 1  & \text{in} \; \partial \Omega_{int} & \text{and} &  \theta  = 0 & \text{in} \; \partial \Omega_{ext}.
     \end{array}
    \right.
    \end{equation*}
Then, $\nabla^\perp \theta$ verifies \eqref{eq.potentiel.rot.div.norm} with $B=0$ and we have in polar coordinates, for all $(r,s) \in [\rho_1 , \rho_2 ] \times [0, 2\pi [$,
$$
\nabla^{\perp} \theta (r,s) = \frac{1}{r\; \mathrm{ln}(\rho_1 / \rho_2 )} \begin{pmatrix} -\sin(s) \\ \cos(s) \end{pmatrix} \; .
$$

Moreover
$$
\int_{\partial \Omega_{int}} \nabla^\perp \theta = \frac{2 \pi}{\ln{\left( \rho_1 / \rho_2 \right)}},
$$
with $\partial \Omega_{int}$ defined in \eqref{eq.Omega}.
    \end{lemma}
We have a family of vector potentials that give rise to unitarily equivalent operators.
    \begin{notation}
We note $\mathrm{arg}(\cdot)$ the principal value of the argument that lies within the interval $[0,2\pi[$.
    \end{notation}
    \begin{proposition} \label{propo.potentiel.flux}
Recall that $\phi$ is defined in \eqref{eq.phi} and $\theta$ in Lemma \ref{lemme.theta}.\\
For all $p\in \ZZ$ and $h >0$, consider
$$
{\bf A_{h,p}} = \nabla^{\perp} \phi \; + \;  h \, \gamma_{h,p} \; \mathrm{ln} \autopar{\frac{\rho_1}{\rho_2}} \, \nabla^{\perp} \theta \, ,
$$
with $\gamma_{h,p} = p + c_0 /h$ and $c_0 =  \rho_1 \partial_r \phi \autopar{\rho_1 } -  \frac{1}{2\pi}\int_{\partial \Omega_{int}} A$.\\ 
Then, we have
$$
\lc \sigma \cdot \left( \mathbf{p} - \mathbf{A_{h,p}} \right) \rc^2 = e^{ip\, \mathrm{arg}} \, \lc \sigma \cdot \left( \mathbf{p} - \mathbf{\check{A}} \right) \rc^2 \, e^{-ip\, \mathrm{arg}}.
$$
    \end{proposition}
    \begin{proof} ~~
    \begin{enumerate}
    \item Let us show that there exists $c\in \RR$ such that $\mathbf{ \check{A} }  = \nabla^{\perp} \phi +  c \nabla^{\perp} \theta$. \\
Let $\alpha \in \RR $ to be determined, consider
$$
F = \mathbf{\check{A}} - \nabla^\perp \phi - \alpha \nabla^\perp \theta .
$$
The vector field $F$ verifies \eqref{eq.potentiel.rot.div.norm}. By linearity, we have $\mathrm{rot} \left( F \right) = \mathrm{div} \left( F \right) = 0$. Moreover, since $\phi$ and $\theta$ are constant on each connected component of the boundary we have
$$
\nabla^\perp \phi \cdot \mathbf{n} = - \frac{\diff}{\diff t} \phi (\gamma(t)) = 0 \; \; \; \text{and similarly} \; \; \; \nabla^\perp \theta \cdot \mathbf{n} = 0,
$$
with $\gamma$ a parametrization of a connected component of the boundary and $\mathbf{n} = \begin{pmatrix} 
0 & 1 \\
-1 & 0 
\end{pmatrix} \gamma'$ the associated unit normal. \\
If we find $\alpha$ such that $\int_{\partial \Omega_{int}} F =0$, then by Lemma \ref{lemme.fct.nulle} we have $F=0$.\\
On the one hand we have for $(r,s) \in [\rho_1 , \rho_2] \times [0, 2 \pi [$
$$
\nabla^{\perp} \phi (r,s) = \partial_r \phi (r) \begin{pmatrix} -\sin(s) \\ \cos(s) \end{pmatrix} ,
$$
thus
$$
\int_{\partial \Omega_{int}} \nabla^\perp \phi = 2\pi \rho_1 \partial_r \phi (\rho_1) .
$$
On the other hand,
$$
\int_{\partial \Omega_{int}} \mathbf{\check{A}} = \int_{\partial \Omega_{int}} A + \nabla f = \int_{\partial \Omega_{int}} A \, .
$$
Finally, the choice
$$
\alpha = \ln{ \left( \frac{\rho_2 }{ \rho_1 } \right)}  \left( \rho_1 \partial_r \phi (\rho_1 ) - \frac{1}{2\pi}  \int_{\partial \Omega_{int}} A  \right),
$$
gives the desired result.
    \item Let us now take $p \in \ZZ$ and $h >0$. \\
For $(x,y) \in \Omega$, the function $\mathrm{exp} (ip \, \mathrm{arg}(x,y))$ is smooth. We can change the magnetic potential by conjugating with the unitary operator $\mathrm{exp} (ip \, \mathrm{arg})$.\\
Note that for all $(r,s) \in \RR_+^* \times [0, 2\pi [$
$$
-ih \nabla = -ih \,\left( \begin{pmatrix} \cos{(s)} \\ \sin{(s)} \end{pmatrix} \partial_r + \frac{1}{r} \begin{pmatrix} - \sin{(s)}  \\ \cos{(s)} \end{pmatrix} \partial_s  \right) \; \; \text{and} \; \; e^{ ip \, \mathrm{arg}(x,y)} = e^{ips} \, .
$$
For $(x,y) \in \Omega$, we have, in polar coordinates, for all $(r,s) \in [\rho_1 , \rho_2 ] \times [0, 2\pi [$,
$$
-ih \nabla  \, e^{ ip \, \mathrm{arg}(x,y)} = \frac{hp}{r} \begin{pmatrix} - \sin{(s)}  \\ \cos{(s)} \end{pmatrix} e^{ ip s}  = hp \, \ln{\left(\frac{\rho_1}{\rho_2} \right)} \nabla^\perp \theta \, e^{ ip \, \mathrm{arg}(x,y)} ,
$$
where we used Lemma \ref{lemme.theta}.\\
Finally, we have 
$$
\lc \sigma \cdot \left( \mathbf{p} - \mathbf{A_{h,p}} \right) \rc^2 = e^{ip\, \mathrm{arg}} \, \lc \sigma \cdot \left( \mathbf{p} - \mathbf{\check{A}} \right) \rc^2 \, e^{-ip\, \mathrm{arg}}
$$
with
$$
{\bf A_{h,p}} = {\bf \check{A}} \, + \, ih \, e^{ip \mathrm{arg}} \,  \nabla e^{-ip \mathrm{arg}} = {\bf \check{A}} \, + \, hp \, \ln{\left(\frac{\rho_1}{\rho_2} \right)} \nabla^\perp \theta.
$$
    \end{enumerate}
    \end{proof}
    \begin{remark} Taking $p=0$, in Proposition \ref{propo.potentiel.flux}, we see that $\mathbf{\check{A}} = \mathbf{A_{h,0}}$.\\ Thus, we have
$$
{\bf \check{A}} = \nabla^{\perp} \phi \; + \; c_0 \; \mathrm{ln} \autopar{\frac{\rho_1}{\rho_2}} \, \nabla^{\perp} \theta \, ,
$$
with $c_0 =  \rho_1 \partial_r \phi \autopar{\rho_1 } -  \frac{1}{2\pi}\int_{\partial \Omega_{int}} A$.
    \end{remark}
    \begin{remark}
The unit operator $\mathrm{exp} (ip\, \mathrm{arg})$ of Proposition \ref{propo.potentiel.flux} is an explicit version of the one given in \cite[Proposition 2.1.3]{fournais2010spectral}. By Lemma \ref{lemme.theta}, it is easy to see that the composition by $\mathrm{exp} (ip\, \mathrm{arg})$ modifies the circulations of the magnetic potential, \ie
$$
\int_{\partial \Omega_{int}} \mathbf{A_{h,p}} = \int_{\partial \Omega_{int}} \mathbf{\check{A}} + 2\pi h p.
$$
    \end{remark}
            \section{Fibration of the Dirichlet-Pauli operator} \label{subsection.fibration_pauli}
In this section, we decompose the Dirichlet-Pauli operator, with potential $\mathbf{A_{h,p}}$, into Fourier series. \\
Under the assumption that the magnetic field is radial, the Dirichlet-Pauli operator in polar coordinates, denoted by $\widetilde{\cP_h}$, acting on $L^2 \left( ]\rho_1, \rho_2 [\times [0,2\pi[,\CC^2 ; r \diff r \right)$ as
    \begin{equation} \label{eq.Pauli.polaire.avec_r}
    \widetilde{\cP_h} = \lc -h^2 \autopar{\partial_{rr}^2 \; + \; \frac{1}{r} \partial_r } \; + \; \autopar{ h \;  \frac{\autopar{-i\partial_s - \gamma_{h,p} } }{r} - \partial_r \phi(r) }^2 \rc I_2 \; - \; hB(r) \sigma_3 ,
    \end{equation}
with $\gamma_{h,p}$ defined in Proposition \ref{propo.potentiel.flux}. \\
Details are given in Appendix \ref{annexe_decompo}.

\smallskip

Thanks to the change of function $u (r) = \sqrt{r} \; v(r)$, we get a new operator acting now on $L^2 \autopar{ ]\rho_1, \rho_2 [ \times [0,2\pi[,\CC^2; \diff r \diff s}$ as
    \begin{equation} \label{eq.Pauli.polaire.sans_r}
    \widehat{\cP_h} = \lc -h^2 \autopar{\partial_{rr}^2 \; + \; \frac{1}{4r^2}} \; + \; \autopar{ h \;  \frac{\autopar{-i\partial_s - \gamma_{h,p} } }{r} - \partial_r \phi(r) }^2 \rc I_2 \; - \; hB(r) \sigma_3 .
    \end{equation}

\medskip

Consider $\cF$, the Fourier isomorphism between $L^2 \autopar{ ]\rho_1, \rho_2 [ \times [0,2\pi[,\CC^2; \diff r \diff s}$ and \\$\ell^2 \left( L^2 \left( ]\rho_1, \rho_2 [,\CC^2; \diff r \right) \right)$. Equation \eqref{eq.Pauli.polaire.sans_r} ensures that $\widehat{\cP_h}$ and $\cF$ commute. We have the following diagram 
    \begin{equation}
       \xymatrix{\relax
         H_{0}^1  \cap H^2 \autopar{]\rho_1, \rho_2 [\times [0,2\pi[,\CC^2; \diff r \diff s} \ar[r]^-{\cF} \ar[d]_-{\widehat{\cP_{h}}}  & \ell^2 \autopar{H_{0}^1  \cap H^2 \autopar{ ]\rho_1, \rho_2 [,\CC^2; \diff r}} \ar[d]^-{\bigoplus_{m\in \ZZ} \cP_{h,m} }   \\
         L^2 \autopar{]\rho_1, \rho_2 [\times [0,2\pi[,\CC^2; \diff r \diff s} \ar[r]_-{\cF}  & \ell^2 \autopar{L^2 \autopar{ ]\rho_1, \rho_2 [,\CC^2; \diff r}} 
      } \, ,
    \end{equation}
where the operator $\cP_{h,m}$ is defined as follows. 
    \begin{definition} \label{def.Pauli.moment.fix}
Let $h \in ]0,1]$ and $m,p \in \ZZ$.
We define the operator $\left( \cP_{h,m} , \mathrm{Dom} \left( \cP_{h,m} \right) \right)$ on $L^2 \autopar{[\rho_1, \rho_2], \CC^2}$ as the operator acting as
    \begin{equation} \label{eq.Phm.def.1}
    \cP_{h,m} = \lc -h^2 \autopar{\partial_{rr}^2 \; + \; \frac{1}{4r^2}} \; + \; \autopar{ h \;  \frac{\autopar{m - \gamma_{h,p} } }{r} - \partial_r \phi(r) }^2 \rc I_2 \; - \; hB(r) \sigma_3 \;  ,   
    \end{equation}
with $\gamma_{h,p}$ defined in Proposition \ref{propo.potentiel.flux}. \\
Moreover, $\mathrm{Dom} \left( \cP_{h,m} \right) = H_{0}^1  \cap H^2 \autopar{ [\rho_1, \rho_2 ],\CC^2}$. We also note $\cL^{-}_{h,m}$ and $\cL^{+}_{h,m} $, the operators given by 
$$
\cP_{h,m} =\begin{pmatrix}
\cL^{-}_{h,m} & 0 \\
0 & \cL^{+}_{h,m}
\end{pmatrix} \,.
$$
    \end{definition}
We give in the next lemma the connection between the spectrum of $\cP_h$ and that of $\cP_{h,m}$ for $m\in \ZZ$. A proof of this result is given in Appendix \ref{annexe_decompo}.
    \begin{lemma} \label{lemme.union.sp}
Under the assumptions made in Section \ref{section.introduction}, we have for all $h \in ]0,1]$ : 
\begin{equation}
\mathrm{Sp} \autopar{ \cP_h } = \displaystyle{\bigcup_{m\in \ZZ} \mathrm{Sp} \autopar{\cP_{h,m}}},    
\end{equation}
with $\cP_{h,m}$ given in Definition \ref{def.Pauli.moment.fix}.\\
Moreover, if $v_m$ is eigenfunction of $\cP_{h,m}$ then $\autopar{r,s} \mapsto \sqrt{r} \; v_m (r) e^{im s}$ is an eigenfunction of $\widetilde{\cP_h }$ associated with the same eigenvalue, with $\widetilde{\cP_h}$ given in \eqref{eq.Pauli.polaire.avec_r}.
    \end{lemma}
    \begin{remark} \label{remark.union.sp}
In the same way as in Lemma \ref{lemme.union.sp}, we have
$$
\mathrm{Sp} \left( \cL_{h}^- \right) =  \displaystyle{\bigcup_{m\in \ZZ} \mathrm{Sp} \left( \cL_{h,m}^- \right) } \; \; \text{and} \; \; \mathrm{Sp} \left( \cL_h^+ \right) = \displaystyle{\bigcup_{m\in \ZZ} \mathrm{Sp} \left( \cL_{h,m}^+ \right) } .
$$
    \end{remark}
        \subsection{Extension of the definition of known operators} 
The definition of $\cP_{h,m}$, see Definition \ref{def.Pauli.moment.fix}, naturally leads us to consider real moments (and not only integer moments). Indeed, in \eqref{eq.Phm.def.1}, the factor $m - \gamma_{h,p}$ is no more an integer. \\
Section \ref{subsection.op_moments_reels} is devoted to the study of some remarkable operators with real angular momentum.
    \subsubsection{Real angular momentum operators} \label{subsection.op_moments_reels}
All the results of this section will be demonstrated in the case where the magnetic potential is equal to $\textbf{A}_0 = \nabla^\perp \phi$. We will see in the next subsection that it is sufficient when considering real moments.

\medskip

Let us define the magnetic gradient, magnetic Laplacian and magnetic Dirac operators appearing during the manipulation of the Pauli operator. This definition allows us to have a better overview of the operator $\cP_{h,m}$.
    \begin{remark} \label{remarque.prelude.moment.reels}
Let $h \in ]0,1]$ and $m \in \ZZ$. We will show that by defining $\cP_{h,m}$ with the potential $\mathbf{A_0}$, we have 
    \begin{equation*}
    \cL_{h,m}^{-} = d_{h,m} d_{h,m}^\times \, , \; \; \; \cL_{h,m+1}^{+} = d_{h,m}^\times d_{h,m} ,
    \end{equation*}
where $d_{h,m}$ given in Definition \ref{definition.extension.op}.
    \end{remark}
    \begin{definition} \label{definition.extension.op}
Let $m \in \mathbb{R}$ and $h>0$.\\
We define $\autopar{\mathbf{p}_{h,m} , \, \mathrm{Dom} \autopar{\mathbf{p}_{h,m}}}, \; \autopar{d_{h,m} , \, \mathrm{Dom} \autopar{d_{h,m}}} , \; \autopar{\mathscr{M}_{h,m} , \, \mathrm{Dom} \autopar{\mathscr{M}_{h,m}}}$ the operators on $L^2 \autopar{[\rho_1, \rho_2], \CC}$ acting as 
    \begin{equation*}
        \begin{split}
        \textbf{p}_{h,m}  &= -ihe_r\left(
            \partial_r-\frac{1}{2r}
        \right)
        +
        e_s\left(
        \frac{hm}{r}-\partial_r\phi
        \right)\,,
\\
        d_{h,m} &= -ih \autopar{\partial_r + \frac{m + 1/2}{r} - \frac{\partial_r\phi}{h} }\,,
\\
        \mathscr{M}_{h,m}
        &=
        \textbf{p}_{h,m}^\times\cdot\textbf{p}_{h,m}
        =
        -h^2\left(
            \partial_r^2+\frac{1}{4r^2}
        \right)
        +
        \left(
        \frac{hm}{r}-\partial_r\phi
        \right)^2\,,
        \end{split}
    \end{equation*}
where the family of vectors $(e_r, e_s)$ constitutes a direct orthonormal basis of $\mathbb{R}^2$ and $\textbf{p}_{h,m}^\times$ is the formal adjoint of $\textbf{p}_{h,m}$.\\
Moreover, 
$$
\mathrm{Dom} \autopar{d_{h,m}} =\mathrm{Dom} \autopar{\textbf{p}_{h,m}}  = H^1_0 \autopar{[\rho_1, \rho_2], \CC} \; \; \text{and} \; \; \; \mathrm{Dom} \autopar{\mathscr{M}_{h,m}} = H^1_0 \cap H^2 \autopar{[\rho_1, \rho_2], \CC}.
$$
    \end{definition}
The operators $\textbf{p}_{h,m},\, d_{h,m},\, \mathscr{M}_{h,m}$ are actually related to each other as one can see in the following.
    \begin{proposition} \label{proposition.extension.operateur}
Let $\textbf{A}_0 = \nabla^\perp \phi$ where $\phi$ is the unique solution of the Poisson's equation \eqref{eq.phi}.

For $m\in\mathbb{R}$ and $h>0$, we have
    \begin{equation} \label{eq.relation.dhm-Mhm}
        \begin{split}
        d^\times_{h,m}d_{h,m} &= \mathscr{M}_{h,m+1}+hB\,,
\\
        d_{h,m}d_{h,m}^\times &= \mathscr{M}_{h,m}-hB\,,
        \end{split}
    \end{equation}
where $d^\times_{h,m}$ is the formal adjoint of $d_{h,m}$ (in the distribution meaning).

When $m$ is an integer, the operators $\textbf{p}_{h,m}$ and $d_{h,m}$ acting on the radial functions verify
    \begin{equation}\label{eq.magnetic_grad_laplacien}
        \begin{split}
        \textbf{p}_{h,m}  &= e^{-ims}r^{1/2}\left(\textbf{p}-\textbf{A}_0\right)r^{-1/2}e^{ims}\,,
\\
        \mathscr{M}_{h,m}
        &=
        e^{-ims}r^{1/2}\left|\mathbf{p}-\mathbf{A_0}\right|^2r^{-1/2}e^{ims}\,,
        \end{split}
    \end{equation}
and
    \begin{equation} \label{eq.extension_op-dirac}
        \begin{split}
        \begin{pmatrix}
        0&d_{h,m}\\d^\times_{h,m}&0
        \end{pmatrix}
        &=
        \begin{pmatrix}
        e^{-ims}&0\\0&e^{-i(m+1)s}
        \end{pmatrix}r^{1/2}\sigma\cdot(\mathbf{p}-\mathbf{A_0})r^{-1/2}\begin{pmatrix}
        e^{ims}&0\\0&e^{i(m+1)s}
        \end{pmatrix}\,,
        \end{split}
    \end{equation}
where  $\textbf{p} = -ih\nabla = -ihe_r\partial_r -ihe_s\frac{\partial_s}{r}$ is the momentum operator in polar coordinates $(e_r(s) = (\cos(s),\sin(s))$, $e_s(s) = (-\sin(s),\cos(s)))$, $\left|\mathbf{p}-\mathbf{A_0}\right|^2$ is the magnetic Laplacian. 
    \end{proposition}
    \begin{proof} Let $u\in \cC^{\infty}_c \autopar{\RR , \CC}$ and $v=(v_1,v_2)\in \cC^{\infty}_c \autopar{\RR , \CC^2}$.
\begin{enumerate}[label = (\roman*)]
\item Let us start with \eqref{eq.relation.dhm-Mhm}. We have
    \begin{align*}
d^\times_{h,m}d_{h,m}&\,u = - h^2  \autopar{\partial_r - \frac{m + 1/2}{r} + \frac{\partial_r \phi}{h}}\autopar{\partial_r + \frac{m + 1/2}{r} - \frac{\partial_r \phi}{h}}\,u \\
& = -h^2 \autopar{\partial_{r}^2 - \frac{m + 1/2}{r^2} - \frac{\partial^{2}_r \phi}{h} - \autopar{\frac{m+1/2}{r}}^2 + 2\frac{m+1/2}{r}\frac{\partial_r \phi}{h} - \autopar{\frac{\partial_r \phi}{h}}^2}u \,,
    \end{align*}
and then
    \begin{equation*}
        \begin{split}
        d^\times_{h,m}d_{h,m}\,u & =-h^2 \left( \partial^{2}_r - \frac{(m+1)-1/2}{r^2}- \frac{\partial^{2}_r \phi}{h} - \autopar{\frac{(m+1)-1/2}{r}}^2 \right. \\
         &  \left. + 2\frac{(m+1)-1/2}{r}\frac{\partial_r \phi}{h} - \autopar{\frac{\partial_r \phi}{h}}^2 \right) u \\
         & = \mathscr{M}_{h,m+1}u + h \autopar{\partial_r^2 \phi + \frac{1}{r}\partial_r \phi} u \\
         & = \autopar{ \mathscr{M}_{h,m+1} + hB } u \,.
        \end{split}
    \end{equation*}
In the same way,
    \begin{equation*}
        \begin{split}
        d_{h,m}d^\times_{h,m}&\,u = - h^2  \autopar{\partial_r + \frac{m + 1/2}{r} - \frac{\partial_r \phi}{h}}\autopar{\partial_r - \frac{m + 1/2}{r} + \frac{\partial_r \phi}{h}}\,u \\
        & = -h^2 \autopar{\partial_{r}^2 + \frac{m + 1/2}{r^2} + \frac{\partial^{2}_r \phi}{h} - \autopar{\frac{m+1/2}{r}}^2 + 2\frac{m+1/2}{r}\frac{\partial_r \phi}{h} - \autopar{\frac{\partial_r \phi}{h}}^2}u \\
         & = \mathscr{M}_{h,m}u - h \autopar{\partial_r^2 \phi + \frac{1}{r}\partial_r \phi} u \\
         & = \autopar{ \mathscr{M}_{h,m} - hB } u \,.
        \end{split}
    \end{equation*}
    \item Let us then establish \eqref{eq.magnetic_grad_laplacien}. On the one hand, we have
    \begin{equation} \label{eq.conjugaison.jacobien}
    \left\{
        \begin{aligned} 
        &r^{1/2} \partial_r r^{-1/2} u = \partial_r u -\frac{1}{2r} u \,,
\\
        &r^{1/2} \partial^{2}_r r^{-1/2} u = \partial^{2}_r u -\frac{1}{r} \partial_r u + \frac{3}{4r^2} u \,,
        \end{aligned} 
    \right.
    \end{equation}

on the other hand, we deduce from the radialness of $\phi$ that
    \begin{equation} \label{eq.potentiel_magnetique}
A_0 = \nabla^\perp \phi = \autopar{e_r\partial_r + e_s\frac{\partial_s}{r} }^\perp \phi = e_s\partial_r \phi\,.
    \end{equation}
The magnetic gradient becomes
    \begin{equation*}
        \begin{split}
        e^{-ims}r^{1/2}\left(\textbf{p}-\textbf{A}_0\right)r^{-1/2}e^{ims}v &= e^{-ims}r^{1/2}\left(-ih\autopar{e_r \partial_r + e_s\frac{\partial_s}{r}}-e_s\partial_r\phi\right)r^{-1/2}e^{ims}v 
\\
        & =-ihe_r  r^{1/2}\partial_r r^{-1/2} v + e_s\autopar{\frac{hm}{r}-\partial_r\phi}v 
\\
        & = \textbf{p}_{h,m} v \,.
        \end{split}
    \end{equation*}
Then, the magnetic Laplacian satisfies
$$
\left|\mathbf{p}-\mathbf{A_0}\right|^2 u = - h^2 \Delta \; u + 2ih A_0 \cdot \nabla \; u+ \valabs{A_0}^2 u \,.
$$
In polar coordinates, we have
        \[\begin{split}
        & r^{1/2} \Delta \; r^{-1/2} u = \autopar{\partial^2_r + \frac{1}{4r^2} + \frac{\partial^2_s}{r^2}} u \,,
\\
        & r^{1/2} A_0 \cdot \nabla \; r^{-1/2} u = \begin{pmatrix} 0 \\ \partial_r \phi \end{pmatrix} \cdot \begin{pmatrix} \partial_r - \frac{1}{2r} \\ \frac{\partial_s}{r} \end{pmatrix} u =\partial_r \phi \frac{\partial_s}{r} u \, ,
\\
        & r^{1/2} \valabs{ A_0 }^2 r^{-1/2} u = \autopar{\partial_r \phi}^2 u \, .
        \end{split}\]
Thus,
    \begin{equation*}
        \begin{split}
        e^{-ims}r^{1/2}\left|\mathbf{p}-\mathbf{A_0}\right|^2r^{-1/2}e^{ims} u &= -h^2  \autopar{\partial^2_r + \frac{1}{4r^2} - \frac{m^2}{r^2}} u - 2 \frac{hm}{r} \partial_r \phi u + \autopar{\partial_r \phi}^2 u\\
        & = \mathscr{M}_{h,m} u \, .
        \end{split}
    \end{equation*}
    \item Let us finish with \eqref{eq.extension_op-dirac}. Using \eqref{eq.potentiel_magnetique} and the properties of Pauli matrices, we have
    \begin{equation*}
        \begin{split}
        \sigma\cdot(\mathbf{p}-\mathbf{A_0}) & = \sigma \cdot \autopar{-ihe_r \partial_r + e_s \lb \frac{-ih \partial_s}{r} - \partial_r \phi \rb} 
\\
        & = \sigma \cdot e_r  \autopar{ -ih \partial_r \Id_2 + i \sigma_3  \lb \frac{-ih \partial_s}{r} - \partial_r \phi  \rb }
\\
        & = \begin{pmatrix}
        0& e^{-is}\\e^{is}&0
        \end{pmatrix} \begin{pmatrix}
        -ih \partial_r + i \lb \frac{-ih\partial_s}{r} - \partial_r \phi \rb & 0 \\ 0 & -ih \partial_r - i \lb \frac{-ih\partial_s}{r} - \partial_r \phi \rb
        \end{pmatrix}
\\
        & = \begin{pmatrix}
        0 & -ihe^{-is} \autopar{\partial_r + \frac{-i\partial_s}{r} - \frac{\partial_r \phi}{h}  } \\ -ih e^{is} \autopar{ \partial_r - \frac{-i\partial_s}{r} + \frac{\partial_r \phi}{h}  } & 0
        \end{pmatrix} \, .
        \end{split}
    \end{equation*}
From Equation \eqref{eq.conjugaison.jacobien}, we have
    \begin{equation*}
      \begin{split}
        r^{1/2}&\sigma \cdot(\mathbf{p}-\mathbf{A_0}) r^{-1/2} \begin{pmatrix}
        e^{ims} & 0\\ 0& e^{i(m+1)s}
        \end{pmatrix} \begin{pmatrix} 
        v_1  \\ 
        v_2 
        \end{pmatrix}  \\
        & = \begin{pmatrix}
        0 & -ihe^{-is} \autopar{\partial_r + \frac{-i\partial_s - 1/2}{r} - \frac{\partial_r \phi}{h}  } \\ -ih e^{is} \autopar{ \partial_r - \frac{-i\partial_s +1/2}{r} + \frac{\partial_r \phi}{h}  } & 0
        \end{pmatrix} \begin{pmatrix} 
        v_1 e^{ims} \\ 
        v_2 e^{i(m+1)s} 
        \end{pmatrix} 
\\
        & = \begin{pmatrix}
        e^{ims} & 0\\ 0& e^{i(m+1)s}
        \end{pmatrix} \begin{pmatrix}
        0&d_{h,m}\\d^\times_{h,m}&0
        \end{pmatrix} \begin{pmatrix} 
        v_1  \\ 
        v_2 
        \end{pmatrix} \,.
        \end{split}
    \end{equation*}
    \end{enumerate}
\end{proof}
The following lemma will be important when studing the adjoint of $d_{h,m}$ and to determine elliptic estimates. 
    \begin{lemma} \label{lemme.calcul_norme2_d_h,m}
Let $h\in]0,1]$, $m\in\mathbb{R}$ and $u\in H^1_0 \autopar{[\rho_1, \rho_2], \CC}$. We have
\[
        \begin{split}
        \left\|d_{h,m}u\right\|^2
        &=
        h^2\left\|\left(\partial_r-\frac{1}{2r}\right)u\right\|^2
        +
        \left\|
        \left(\frac{h(m+1)}{r}-\partial_r\phi\right)u
        \right\|^2
        +
        h\left\|
        \sqrt{B}u
        \right\|^2
\\
        \left\|d^\times_{h,(m+1)}u\right\|^2
        &=
        h^2\left\|\left(\partial_r-\frac{1}{2r}\right)u\right\|^2
        +
        \left\|
        \left(\frac{h(m+1)}{r}-\partial_r\phi\right)u
        \right\|^2
        -
        h\left\|
        \sqrt{B}u
        \right\|^2 \, ,
        \end{split}
    \]
    \end{lemma}
    \begin{proof}
Let $h\in]0,1]$, $m\in\mathbb{R}$ and $u \in H^1_0 \autopar{[\rho_1, \rho_2], \CC}$. By using Proposition \ref{proposition.extension.operateur}, one easily checks that
\begin{align*}
\left\|d_{h,m}u\right\|^2
&= \left\|\textbf{p}_{h,m+1} u\right\|^2 + h \left\|\sqrt{B} \, u\right\|^2 \\
&= h^2\left\|\left(\partial_r-\frac{1}{2r}\right)u\right\|^2
+
\left\|
\left(\frac{h(m+1)}{r}-\partial_r\phi\right)u
\right\|^2
+
h\left\|
\sqrt{B}u
\right\|^2 \,.
\end{align*}
We proceed in the same way for $\left\|d^\times_{h,(m+1)} u\right\|^2$.
    \end{proof}
Next proposition gives an explicit description of the kernel of the adjoint of $d_{h,m}$ and of its orthogonal. 
    \begin{proposition} ~~
    \label{propo.Dirac}
\begin{enumerate}[label = (\roman*)]
    \item The operator $\autopar{d_{h,m} , \; \mathrm{Dom} \autopar{d_{h,m}}}$ is closed with closed range.
    \item The adjoint $\autopar{d_{h,m}^* , \; \mathrm{Dom} \autopar{d_{h,m}^*}}$ acts as $d_{h,m}^{\times}$ on $\mathrm{Dom} \autopar{d_{h,m}^*} = H^1 \autopar{[\rho_1, \rho_2], \CC} $ 
    and 
    $$
    \mathrm{ker} \autopar{d_{h,m}^*} = \mathrm{Vect} \autopar{r \mapsto e^{- \phi /h} r^{m +1/2}}. 
    $$
    \item We have, $\mathrm{ker} \autopar{d_{h,m}^*}^{\perp} \cap \mathrm{Dom} \autopar{d_{h,m}^*} = \left\{ d_{h,m} w ; \; \; w \in H_0^1 \cap H^2 \autopar{[\rho_1, \rho_2], \CC} \right\}$.
\end{enumerate}
    \end{proposition}
    \begin{proof} ~ \\
        \begin{enumerate}[label = (\roman*)]
            \item Lemma \ref{lemme.calcul_norme2_d_h,m} ensures the equivalence between the norm $H_0^1 \autopar{[\rho_1, \rho_2], \CC}$ and the graph norm of $d_{h,m}$. Thus, the operator $\autopar{d_{h,m} , \; \mathrm{Dom} \autopar{d_{h,m}}}$ is closed. The closed image property follows from the same lemma. Indeed, for any $u \in H_0^1 \autopar{[\rho_1, \rho_2], \CC}$
    \begin{equation*}
    \left\|d_{h,m} u\right\| \geq \sqrt{2 h B_0 } \left\|u\right\|,
    \end{equation*}
    with $B_0 = \inf_{x \in \Omega} B(x) >0$ by assumption.
            \item By definition,
    \begin{equation*}
    \mathrm{Dom} \autopar{d_{h,m}^*} \subset \left\{ u \in L^2 ; \; \; d_{h,m}^\times u \in L^2  \right\} = H^1 \autopar{[\rho_1, \rho_2], \CC} .
    \end{equation*}
    On the other hand, if $v \in H^1 \autopar{[\rho_1, \rho_2], \CC}$ we have for all $w\in \cC^{\infty}_c \autopar{[\rho_1, \rho_2], \CC}$
    \begin{equation*}
    \left\langle v, d_{h,m} w  \right\rangle = \left\langle d_{h,m}^\times v , w  \right\rangle,
    \end{equation*}
    and $d_{h,m}^\times v \in L^2 \autopar{[\rho_1, \rho_2], \CC}$. Finally, we can extend by density this result to $\mathrm{Dom} \autopar{d_{h,m}}$ to obtain $v \in \mathrm{Dom} \autopar{d_{h,m}^*}$ and $d_{h,m}^* = d_{h,m}^\times $.\\
    Moreover,
    \begin{align*}
    \mathrm{ker} \autopar{d_{h,m}^*} & = \left\{ u \in L^2 ; \; \;  d_{h,m}^\times u =0 \right\} \\
    & = \left\{ e^{-\phi/h} v , v \in L^2 ; \; \;  \autopar{\partial_r - \frac{m+1/2}{r}} v =0 \right\} \\
    & = \mathrm{Vect} \autopar{r \mapsto e^{- \phi /h} r^{m +1/2}}.
    \end{align*}
            \item The following equalities are consequences of Proposition \ref{proposition.extension.operateur},
    \begin{equation*}
    \begin{split}
    &\mathrm{ker} \autopar{d_{h,m}^*}^{\perp} \cap \mathrm{Dom} \autopar{d_{h,m}^*} = \mathrm{Im} \autopar{d_{h,m}} \cap \mathrm{Dom} \autopar{d_{h,m}^*} \\
    & = \left\{ d_{h,m} w ; \; \; w \in H_0^1   \text{ and } d_{h,m}^\times d_{h,m} w = \autopar{\cM_{h,m+1} + hB}w \in L^2 \right\} \\
    & = \left\{ d_{h,m} w ; \; \; w \in H_0^1 \cap H^2 \autopar{[\rho_1, \rho_2], \CC} \right\} .
    \end{split}
    \end{equation*}
    \end{enumerate}
\end{proof}
        \subsection{Back to the Pauli operator and symmetry of the Dirac operator}
Following the same idea as in Remark \ref{remarque.prelude.moment.reels}, we can now extend the definition of the Pauli operator for real moments with a magnetic potential equal to $\mathbf{A_{h,p}}$, with $p \in \ZZ$ and $h \in ]0,1]$. The main result of this part is Proposition \ref{corollaire.minmax.mom_reel} that give us an explicit formula for the spectrum of $\cL_{h,m}^-$.
    \begin{definition} \label{definition.sp_lambda_hm}
Let $h>0$, $m \in \RR$ and $p\in \ZZ$.\\
Consider $\left( \cL_{h,m}^{\pm} , Dom\left( \cL_{h,m}^{\pm} \right) \right),$ the self-adjoint operators acting as
    \begin{equation*}
\cL_{h,m}^{-} = d_{h,\widetilde{m}} d_{h,\widetilde{m}}^\times \, , \; \; \; \cL_{h,m+1}^{+} = d_{h,\widetilde{m}}^\times d_{h,\widetilde{m}} \; \; and \; \; Dom\left( \cL_{h,m}^{\pm} \right) = H^1_0 \cap H^2 \autopar{[\rho_1, \rho_2], \CC},
    \end{equation*}
where $d_{h,\widetilde{m}}$ defined in Definition \ref{definition.extension.op}, with $\widetilde{m} = m - \gamma_{h,p}$ and $\gamma_{h,p}$ defined in Proposition \ref{propo.potentiel.flux}. Moreover, we let $\left( \lambda_{k,\widetilde{m}} (h) \right)_{k\in \NN^{*}}$ be the non-decreasing sequence of eigenvalues associated to $\cL^{-}_{h,m}$.
    \end{definition}
    \begin{remark}
By using the min-max theorem, we have
$$
\lambda_{k,m} (h) = \underset{\underset{\mathrm{dim}V = k}{V \subset \mathrm{Dom}\autopar{d_{h,m}^{\times}}}}{\mathrm{inf}} \; \underset{v\in V\backslash \{ 0 \} }{\mathrm{sup}} \frac{\left\| d_{h,m}^{\times} \, v \right\|^2_2}{\left\| v \right\|^2_2} .
$$
with $d_{h,m}^\times$ defined in Definition \ref{definition.extension.op}.
    \end{remark}
The next result (see \cite[Lemma 3.11]{barbaroux:hal-01889492}) ensures that it is possible to modify the magnetic field of the Dirac operator by composing with well chosen matrices. This explains the choice, made in subsection \ref{subsection.op_moments_reels}, to define the Dirac operator for the magnetic potential $\mathbf{A_0} = \nabla^\perp \phi$. We invite the reader to consult the paper \cite{barbaroux:hal-01889492} for a proof of the next lemma.
    \begin{lemma}
    \label{lemme.jauge}
We have the following relation
\begin{equation}
    e^{\sigma_3 \frac{\Psi}{h}} \; \; \sigma \cdot {\bf p} \; \; e^{\sigma_3 \frac{\Psi}{h}} = \sigma \cdot \left( {\bf p } - \nabla\Psi^{\perp}  \right),
\end{equation}
with $\Psi \in \cC^1 \autopar{\overline{\Omega}}$ and $\sigma_3$ defined in \eqref{eq.matrices-Pauli}.
    \end{lemma}
We use this lemma to make the change of function $v = e^{-\phi/h} u$.
    \begin{proposition} \label{corollaire.minmax.mom_reel}
Let $k \in \NN^* $, $m \in \RR$ and $h\in ]0,1]$. We have
    \begin{equation} \label{eq.minmax.moment.reels}
        \lambda_{k,m} (h) = \underset{\underset{\mathrm{dim}V = k}{V \subset H_{0}^{1} \left([\rho_1 , \rho_2], \CC \right)}}{\mathrm{inf}} \; \underset{v\in V\backslash \{ 0 \} }{\mathrm{sup}} \; \frac{\displaystyle{h^2 \int_{\rho_{1}}^{\rho_{2}}  e^{-2\phi / h} \valabs{\autopar{ \partial_{r} - \frac{m + 1/2}{r}} v(r)}^2 \diff r }}{ \displaystyle{\int_{\rho_{1}}^{\rho_{2}} e^{-2\phi /h} \valabs{v(r)}^2  \diff r} }.
    \end{equation}
    \end{proposition}
            \section{Proof of Theorem \ref{theorem.principal}} 
In this section we establish Theorem \ref{theorem.principal}. 
    \begin{notation}
Let $j\in \NN^*$, $h > 0$ and $m \in \RR$, we consider
\[f_{j,h} (m) = \frac{\lambda_{j,m} (h)}{\sqrt{h}} e^{-2 \phi_{\min}/h}  \,.\]
    \end{notation}
The following proposition will be proved in Sections \ref{sec:majoration} and \ref{section.preuve.propo} (for the first two points) and Section \ref{section.estim.elliptiques} (for the last point).
    \begin{proposition} \label{propo.CVU.coercivite}
We have 
\begin{enumerate}
\item \label{pt1.propCVU_coercivite} $\autopar{f_{1,h} }_{h \in ]0, 1] }$ is uniformly convergent on any compact to $f\colon \RR \rightarrow \RR^+$ when $h$ tends to $0$ where
\begin{equation} \label{eq.f}
f\colon m \longmapsto 2 \sqrt{\frac{\phi''_{\min}}{\pi}} \autopar{\partial_n\phi\autopar{\rho_1} \autopar{\frac{\rho_1}{r_{\min}} }^{2m+1} + \partial_n\phi\autopar{\rho_2} \autopar{\frac{\rho_2}{r_{\min}} }^{2m+1}  } \, .
\end{equation}
\item \label{pt2.propCVU_coercivite}
Let $K>0$ and $h_0>0$. There exists $g : \RR \rightarrow \RR^+ $ coercive such that for all $h\in ]0, h_0] \, ,$ and $m \in \RR$ satisfying
\[
f_{1,h}(m)\leq K\,,
\]
one has
$f_{1,h} (m) \geq g (m)$.
\item  \label{pt3.propCVU_coercivite} 
For all $h>0$, $j\geq 2$ and $m\in \RR$, we have
\[
f_{j,h}(m)\geq 2B_0 \sqrt{h}e^{-2\phi{\rm min}/h}\,.
\]
\end{enumerate}
    \end{proposition}
    \begin{remark} \label{remark.min.f}
For all $m \in \RR\, ,$
$$
f(m) \geq c >0 \, ,
$$
with $c = 2 \sqrt{\frac{\phi''_{\min}}{\pi}} \min \, \autopar{\partial_n\phi\autopar{\rho_1}, \partial_n\phi\autopar{\rho_2}}$.
    \end{remark}
Thus, let us temporarily admit that Proposition \ref{propo.CVU.coercivite} holds.

\medskip

Now we can use Proposition \ref{propo.potentiel.flux} and choose, for each $h>0$, a good magnetic potential. For all $h>0 \, ,$ we consider the Dirichlet-Pauli operator associated to the vector potential 
$$
\mathbf{A_h} = \nabla^{\perp} \phi + h \gamma(h) \mathrm{ln} \autopar{\frac{\rho_1}{\rho_2}} \nabla^{\perp} \theta \; ,
$$
with $\gamma(h) = \frac{c_0}{h} - \left\lfloor \frac{c_0}{h} \right\rfloor$ in such a way that $\gamma(h) \in [0,1]$. \\
Definition \ref{definition.sp_lambda_hm} and Remark \ref{remark.union.sp} ensure that
\begin{equation} \label{eq.spectre_de_l_moins}
\displaystyle{
Sp \autopar{\cL^{-}_h} = \lb \lambda_{j,\, m - \gamma(h)} (h)\,, m\in\ZZ\,,j\in \NN^* \rb .
}
\end{equation}
The following lemma justifies the existence of the prefactor given in Theorem \ref{theorem.principal}.
    \begin{lemma}\label{lem.deflimf}
Let $h>0$, $k\in\NN^*$, we consider the non-decreasing sequence
\begin{equation}\label{eq.minmaxf}
\alpha_k(h) = \inf_{\tiny\begin{array}{c}V\subset\ZZ\\\# V = k\end{array}}\sup_{m\in V}f(m-\gamma(h))\,.
\end{equation}
Then, there exists a $k$-tuple
\[
V_k(h) = \{\mu_1(h),\dots,\mu_k(h)\} 
\]
which realizes the infimum. Moreover, the functions $h\mapsto \alpha_k(h)$ and $h\mapsto \mu_j(h)$ are bounded for $j\in\{1,\dots,k\}$.
    \end{lemma}
    \begin{proof}
Let $k\in\mathbb{N}^*$ and $h>0$ fixed. Let $(W_j)_{j\geq 1}$ a minimizing sequence of \eqref{eq.minmaxf}, \ie for $j\in\{1,\dots,k\}$ $W_j$ is a $k$-tuple in $\ZZ$.
Since the function $m\mapsto f(m-\gamma(h))$ is coercive, the sets $W_j$ are uniformly bounded with respect to $j$. There exists a subset of $\mathbb{Z}$ which realizes the infimum. The continuity of $f$ ensures that $h\mapsto \alpha_k(h)$ is bounded.\\
The sequence $\left( \alpha_k (h) \right)_{k\geq 1}$ is non-decreasing, in fact for all $W \subset \ZZ^{k+1}$ we have
$$
\alpha_k (h) = \sup_{m\in V_{k}}f(m-\gamma(h)) \leq \sup_{m\in W}f(m-\gamma(h)).
$$
By taking the infimum on the set of $k$-tuple, we have
$$
\alpha_k (h) \leq \alpha_{k+1} (h).
$$
    \end{proof}
    \begin{proposition}\label{proposition.borne.sup}
Let $k\in\NN^*$, $h>0$, we have
\[
\lambda_k(h)\leq \alpha_k(h)\sqrt{h}e^{2\phi_{\rm min}/h}(1+o_{h\to0}(1))\,.
\]
    \end{proposition}
    \begin{proof}
Let $k\geq 1$ and $h>0$. We have
\[
\lambda_{k}(h) \leq \max_{\mu\in V_k(h)}\lambda_{1,\,\mu-\gamma(h)}(h) = \max_{\mu\in V_k(h)}f ( \mu -\gamma(h) ) \sqrt{h} \, e^{2 \phi_{\min} / h}.
\]
Proposition \ref{propo.CVU.coercivite} (\ref{pt1.propCVU_coercivite}) and the boundedness of $h\mapsto \mu_k(h)$ (Lemma \ref{lem.deflimf}) give the conclusion.
    \end{proof}
    \begin{lemma}\label{lem.deflimf1h}
Let $h>0$, $k\in\NN^*$, consider the non-decreasing sequence
\begin{equation}\label{eq.minmaxf1h}
\beta_k(h) = \inf_{\tiny\begin{array}{c}V\subset\ZZ\\\# V = k\end{array}}\sup_{m\in V}f_{1,h}(m-\gamma(h))\,.
\end{equation}
Assume that there exists $h_0>0$ such that
\[
\sup_{h\in]0,h_0]}\beta_k(h)<+\infty.
\]
Then, for all $h\in]0,h_0]$, there exists a $k$-tuple
\[
W_k(h) = \{m_1(h),\dots,m_k(h)\}\subset\mathbb{Z}
\]
which realizes the infimum. Moreover, the functions $h\mapsto m_j(h)$ are bounded for $j\in\{1,\dots,k\}$.
    \end{lemma}
    \begin{proof}
Let $k\in \NN^*$ fixed. There exists $h_0 >0$ such that
$$
K = \sup_{h \in ]0, h_{0}]} \beta_k (h) < + \infty
$$
\begin{enumerate}[label = \roman*)]
\item \label{point_i_minmanf1h} For all $\varepsilon \in ] 0 , 1 ]$ and for all $h \in ]0, h_0]$, there exists $X_\varepsilon (h) \subset \ZZ$ such that $\sharp X_\varepsilon (h) = k$ and
\begin{equation*}
\beta_k (h) \leq \max_{m \in X_{\varepsilon} (h)} f_{1,h} (m - \gamma (h)) \leq \beta_k (h) + \varepsilon \leq K + 1 \, .
\end{equation*}
Thus, for all $h \in ]0, h_0]$ and for all $m \in \bigcup_{(h , \varepsilon ) \in ]0 , h_{0}] \times ]0,1]} X_\varepsilon (h)$, we have
$$
f_{1,h} (m - \gamma (h)) \leq K +1 \, .
$$
Proposition \ref{propo.CVU.coercivite} (\ref{pt2.propCVU_coercivite}) ensures that there exists $g : \RR \rightarrow \RR^+ $ coercive such that for all $h \in ]0, h_0]$ and for all $m \in \bigcup_{(h , \varepsilon ) \in ]0 , h_{0}] \times ]0,1]} X_\varepsilon (h)$
\begin{equation*}
g(m- \gamma(h)) \leq f_{1,h} (m - \gamma (h)) \, .
\end{equation*}
Thus there exist $M_k \in \NN^*$ independent of $h$ and $\varepsilon$ such that for any $\varepsilon \in ] 0 , 1 ]$ and for any $h \in ]0, h_0]$, $X_\varepsilon (h) \subset [ - M_k, M_k ]$.

\item Let us consider for all $h \in ]0 , h_0]$, $\autopar{ X_{1/n} (h)}_{n \in \NN^{*}}$ the minimizing sequence of \eqref{eq.minmaxf1h}. From the point \ref{point_i_minmanf1h}, the $k$-tuple $X_{1/n} (h)$ are uniformly bounded with respect to $n$ and $h$. Consequently there exists a $k$-tuple, bounded in $h$, which realizes the infimum.  
\end{enumerate}
In the same way as for $\left( \alpha_k (h) \right)_{k\geq 1}$, the sequence $\left( \beta_k (h) \right)_{k\geq 1}$ is non-decreasing, see the proof of Lemma \ref{lem.deflimf}.
    \end{proof}
    \begin{lemma} \label{lemme:j=1}
Let $k\in \NN^* \,,$ there exists $h_0 >0\,,$ such that
\[
\sup_{h\in]0,h_0]}\beta_k(h)<+\infty ,
\]
and for all $h \in ]0,h_0]\, ,$ 
$$
\lambda_k (h) = \max_{m\in W_k(h)}\lambda_{1, m - \gamma(h)} (h) \, .
$$
    \end{lemma}
    \begin{proof}
Let $k\in \NN^*$. 
According to \eqref{eq.spectre_de_l_moins}, there exists $n_k (h) \in \ZZ\,,j_k (h) \in \NN^*$ such that
$$
\lambda_k (h) = \lambda_{j_{k} (h), n_{k} (h) - \gamma (h)} (h) \,.
$$
However, Proposition \ref{propo.CVU.coercivite} (\ref{pt3.propCVU_coercivite}) ensures that for all $h>0$, $j\geq 2$ and $m\in \RR$, 
\[
\lambda_{j,m} (h) \geq 2B_0 h\,.
\]
By using Proposition \ref{proposition.borne.sup} and the boundedness of $h\mapsto \alpha_k(h)$ (Lemma \ref{lem.deflimf}), we have the existence of $h_0>0$ such that for all $h \in ]0, h_0]$, $j_k(h)=1$ then
$$
\lambda_k (h) = \inf_{\tiny\begin{array}{c}V\subset\ZZ\\\# V = k\end{array}}\sup_{m\in V}\lambda_{1,m-\gamma(h)} (h) = \beta_k(h) \sqrt{h} \, e^{2 \phi_{\min} / h} .
$$
Finally, by using Proposition \ref{proposition.borne.sup} and Lemma \ref{lem.deflimf1h}, we have
\[
\sup_{h\in]0,h_0]}\beta_k(h)<+\infty \; \; \text{and} \; \; \lambda_k (h) = \max_{m\in W_k(h)}\lambda_{1, m - \gamma(h)} (h) \, .
\]
    \end{proof}
    \begin{proof}[Proof of Theorem \ref{theorem.principal}] ~~ \\
    \label{demonstration_theoreme}
Let $k \in \NN^*$. From Lemma \ref{lemme:j=1}, we have for all $h \in ]0,h_0]$, 
\begin{equation} \label{eq.proofthm.1}
    \lambda_k (h) = \max_{m\in W_{k}(h)}\lambda_{1, m - \gamma (h)} (h) = \max_{m\in W_{k}(h)}f_{1,h} (m - \gamma(h)) \sqrt{h} \; e^{2\phi_{\min}/h} \,.
\end{equation}
Moreover, Lemma \ref{lem.deflimf1h} states that for any $j \in \llc 1, k \rrc \, ,$
the functions $h \mapsto m_j (h) - \gamma (h)$ are uniformly bounded.\\
Proposition \ref{propo.CVU.coercivite} (\ref{pt1.propCVU_coercivite}) and Remark \ref{remark.min.f} ensures that 

\begin{equation} \label{eq.proofthm.2}
\begin{split}
 \max_{m\in W_k(h)}f_{1,h} (m - \gamma (h)) 
 &=   (1 + o_{h \rightarrow 0} (1))\max_{m\in W_k(h)}f (m - \gamma (h))
 \\
 &\geq (1 + o_{h \rightarrow 0} (1))\alpha_k(h) \,.    
\end{split}
\end{equation}
Using \eqref{eq.proofthm.1} and \eqref{eq.proofthm.2}, we have the following lower bound
$$
\lambda_k (h) \geq \alpha_k(h) \sqrt{h} \; e^{2\phi_{\min}/h} (1 + o_{h \rightarrow 0} (1))\,.
$$
The upper bound of Proposition \ref{proposition.borne.sup} completes the proof.
    \end{proof}
\section{Elliptic estimates} \label{section.estim.elliptiques}
In this section we give elliptic estimates related to the Dirac operator with real fixed angular momentum, see Definition \ref{definition.extension.op}. These results complete the study realized in \cite{barbaroux:hal-01889492} and make it possible to establish, for example, Proposition \ref{propo.CVU.coercivite} (\ref{pt3.propCVU_coercivite}) (cf. Corollary \ref{corollaire_demo_pt3_propo}).
        \subsection{Ellipticity results on the Dirac operator}
The main result of this section is Proposition \ref{proposition_inegalite_ellipticite}. It is one of the main ingredients to prove a monomial approximation of the associated eigenfunctions (in a neighborhood of $\Omega$), see Proposition \ref{propo:approx_holo}. 
    \begin{notation} \label{notation.Idelta}
Let us denote for $\delta \in ( 0, \delta_0) $
$$
\cI_\delta \, =\, [\rho_1 + \delta , \rho_2 - \delta] \, ,
$$
with $\delta_0 = \min \left( (r_{\min} - \rho_1) /2 ,( \rho_2 - r_{\min} ) / 2 \right)$.
    \end{notation}
    \begin{proposition} \label{proposition_inegalite_ellipticite}
There exist $\delta_0 , h_0 >0$ and $C>0$ such that for all $\delta \in ]0, \delta_0]$, for all $h_0 \in ]0, h_0 ]$ and for all $u\in \mathrm{Dom} \left( d_{h,m}^\times \right) \cap \mathrm{ker} \left( d_{h,m}^\times \right)^\perp$
\begin{gather*} 
\left\|d_{h,m}^\times u\right\|_{L^{2} (\cI_{\delta})}^2 \geq 2 h B_0 \left\|u\right\|_{L^{2} (\cI_{\delta})}^2 \\
\|d_{h,m}^\times u\|_{L^{2} (\cI_{\delta})}^2 \geq  C h^{3/2} \left( \partial_n \phi (\rho_1) \left| u \autopar{\rho_1 + \delta} \right|^2 + \partial_n \phi (\rho_2 ) \left| u \autopar{\rho_2 - \delta} \right|^2 \right) \,,
\end{gather*} 
with $\cI_\delta$ defined in Notation \ref{notation.Idelta}.
    \end{proposition}
For convenience, we will use, in this section, the following convention $\| \cdot \| = \| \cdot \|_{L^{2} (\cI_{\delta})}$.

\medskip

We will need the next lemma later. Note that it implies that the spectrum of $\cL^+_{h,m}$ is a subset of $[2hB_0 , + \infty [$.
    \begin{lemma} \label{lemme_ellipticite_spectre_ss_ensemble}
Let $h\in]0,1]$, $m\in\mathbb{R}$ and $u\in H^1_0\cap H^2 \left( \cI_\delta , \CC \right)$, we have
$$
\left\|d^\times_{h,m} d_{h,m} u\right\| \geq 2 h B_0 \left\|u\right\| \; \; \; \text{and} \; \; \; \left\|d^\times_{h,m} d_{h,m} u\right\| \geq \sqrt{2 h B_0} \, \left\| d_{h,m} u \right\| \, .
$$
Moreover, the spectrum of $d^\times_{h,m} d_{h,m}$ is a subset of $[2hB_0 , + \infty [$.
    \end{lemma}
    \begin{proof}
Let $h\in]0,1]$, $m\in\mathbb{R}$ and $u\in H^1_0\cap H^2 \left( \cI_\delta , \CC \right)$. The second identity of Lemma \ref{lemme.calcul_norme2_d_h,m} ensures that
$$
h^2\left\|\left(\partial_r-\frac{1}{2r}\right)u\right\|^2
+
\left\|
\left(\frac{h(m+1)}{r}-\partial_r\phi\right)u
\right\|^2 \geq h \left\|\sqrt{B}u\right\|^2 \, .
$$
Hence
$$
\left< u, d^\times_{h,m} d_{h,m} u \right> \geq 2h B_0 \, \left\|u\right\|^2 \, ,
$$
\ie according to the min-max theorem, the spectrum of $d^\times_{h,m} d_{h,m}$ is a subset of $[2hB_0 , + \infty [$. \\
And therefore
$$
\left\| d^\times_{h,m} d_{h,m} u\right\| \geq 2 h B_0 \left\|u\right\| \, .
$$
Finally,
\begin{equation*}
2 h \, B_0 \, \left\| d_{h,m} u \right\|^2 \leq 2 h \, B_0 \,  \left< u , d^\times_{h,m} d_{h,m} u \right> \leq 2 h \, B_0 \,   \left\|u\right\| \left\| d^\times_{h,m} d_{h,m} u\right\|\leq \left\| d^\times_{h,m} d_{h,m} u\right\|^2 \, .
\end{equation*}
    \end{proof}
Proposition \ref{proposition_inegalite_ellipticite} is a consequence of the following two lemmas, see Proof \ref{proof.inegalite.elliptique}.
    \begin{lemma} \label{lemme_inegalite_inter}
Let $h\in]0,1]$, $m\in\mathbb{R}$ and $u\in H^1_0\cap H^2 \left( \cI_\delta , \CC \right)$. We have
\[
\begin{split}
\|d^\times_{h,m}d_{h,m}u\|^2
&=
\|\mathscr{M}_{h,m+1}u\|^2 + h^2\|Bu\|^2 + 2h\|\sqrt{B}\textbf{p}_{h,m+1}u\|^2  + h^3\int (-\Delta_r B)|u|^2 \diff r\,,
\\
2h\|\sqrt{B}\textbf{p}_{h,m+1}u\|^2
&=
2h^3\left\|\sqrt{B}\left(\partial_r -\frac{1}{2r}\right)u\right\|^2
+
2h\left\|\sqrt{B}\left(\frac{h(m+1)}{r}-\partial_r \phi\right)u\right\|^2\,,
\\
\|\mathscr{M}_{h,m+1}u\|^2
&=h^4\left\|\left(\partial_r^2+\frac{1}{4r^2}\right)u\right\|^2
+2h^2\left\|
\left(\frac{h(m+1)}{r}-\partial_r\phi\right)\left(\partial_r -\frac{1}{2r}\right)u
\right\|^2
\\&\quad
+h^2\int|u|^2\left(-\Delta_r\right)\left(\frac{h(m+1)}{r}-\partial_r\phi\right)^2 \diff r
\\&\quad
+\left\|
\left(
\frac{h(m+1)}{r}-\partial_r\phi
\right)^2 u
\right\|^2\,,
\end{split}
\]
where $\Delta_r = \partial_r^2+r^{-1}\partial_r$.
There exist also $h_0 > 0$ and a constant $C>0$ (independent of $h$ and $m$) such that for all $h\in (0,h_0)$ and $m\in \RR$
\begin{multline} \label{propo4.inegalite}
C\|d^\times_{h,m}d_{h,m}u\| \geq \sqrt{h} \left( \|u\| + \left\| \left(\frac{h(m+1)}{r}-\partial_r \phi\right)u\right\| \right) 
+h^2
\left\|\left(\partial_r^2+\frac{1}{4r^2}\right)u\right\|
\\
+ h \left\|
\left(\frac{h(m+1)}{r}-\partial_r\phi\right)\left(\partial_r -\frac{1}{2r}\right)u
\right\|
+ h^{3/2} \left\|\left(\partial_r -\frac{1}{2r}\right)u\right\| \,.
\end{multline}
    \end{lemma}
    \begin{proof}
Let $h\in]0,1]$, $m\in\mathbb{R}$ and $u\in H^1_0\cap H^2 \left( \cI_\delta , \CC \right)$. According to Proposition \ref{proposition.extension.operateur},
\begin{align}
\label{eq.extension.1}
d^\times_{h,m}d_{h,m} = \mathscr{M}_{h,m+1}+hB \,, \\
\label{eq.extension.2}
\mathscr{M}_{h,m+1}=\textbf{p}_{h,m+1}^\times\cdot\textbf{p}_{h,m+1} \, .
\end{align}
\begin{enumerate}[label=\arabic*)]
\item \label{preuve_propo8_pt1} Equation \eqref{eq.extension.1} ensures that
\begin{equation*}
    \left\|d^\times_{h,m}d_{h,m}u\right\|^2 = \left\|\mathscr{M}_{h,m+1}u\right\|^2 + 2h \mathop{\mathrm{Re}} \left<\mathscr{M}_{h,m+1}u, Bu\right> +h^2 \left\|Bu \right\|^2  \,.
\end{equation*}
However, according to \eqref{eq.extension.2},
\begin{equation*}
    \left<\mathscr{M}_{h,m+1}u, Bu\right> =  \left\|\sqrt{B} \textbf{p}_{h,m+1} u\right\|^2 + \left<\textbf{p}_{h,m+1}u, \lc \textbf{p}_{h,m+1},  B \rc u \right> \,,
 \end{equation*}
with $\lc \textbf{p}_{h,m+1},  B \rc = \textbf{p}_{h,m+1} \, B  - B \, \textbf{p}_{h,m+1}$. \\
Let us estimate the term involving the commutator. We have
$$
\lc \textbf{p}_{h,m+1},  B \rc =  e_r \autopar{-ih \partial_r B} .
$$
Thus,
\begin{align*}
\left<\textbf{p}_{h,m+1}u, \lc \textbf{p}_{h,m+1},  B \rc u \right> &= -h^2 \left< u , \autopar{\partial_r + \frac{1}{2r}}( \partial_r B \, u) \right> \\
& = -h^2 \left<u, \partial_r B \partial_r u \right> -h^2 \left<u, \autopar{\partial_r^2 B + \frac{\partial_r B}{2r}} u \right> .
\end{align*}
By integration by parts,
\begin{equation*}
    2 \mathop{\mathrm{Re}} \left<u, \partial_r B \partial_r u \right>  = - \left<u, \partial_r^2 B  u \right>  \,. 
\end{equation*}
Therefore
\begin{equation*}
2 \mathop{\mathrm{Re}} \left<\textbf{p}_{h,m+1}u, \lc \textbf{p}_{h,m+1},  B \rc u \right> = h^2 \int (-\Delta_r B)|u|^2 \diff r \,, \; \Delta_r = \partial_r + \frac{\partial_r}{r}\, .
\end{equation*}
Let us now expand the terms $\|\sqrt{B}\textbf{p}_{h,m+1}u\|^2$ and $\left\|\mathscr{M}_{h,m+1}u\right\|^2$.\\
By definition of $\textbf{p}_{h,m+1}$ (cf. Proposition \ref{proposition.extension.operateur}), 
$$
\|\sqrt{B}\textbf{p}_{h,m+1}u\|^2
=
h^2\left\|\sqrt{B}\left(\partial_r -\frac{1}{2r}\right)u\right\|^2
+
\left\|\sqrt{B}\left(\frac{h(m+1)}{r}-\partial_r \phi\right)u\right\|^2\,.
$$
Using the definition of $\mathscr{M}_{h,m+1}$ we also have
\begin{multline*}
    \left\|\mathscr{M}_{h,m+1}u\right\|^2 = h^4 \left\|\autopar{\partial^2_r + \frac{1}{4r^2}}u\right\|^2 +  \left\|\autopar{\frac{h(m+1)}{r} - \partial_r \phi}^2 u\right\|^2+ \\ + 2h^2 \mathop{\mathrm{Re}} \left<-\autopar{\partial^2_r + \frac{1}{4r^2}}u, \autopar{\frac{h(m+1)}{r} - \partial_r \phi}^2 u\right>   \,.
\end{multline*}
Then, noticing $-\autopar{\partial^2_r + \frac{1}{4r^2}} = \autopar{-\partial_r -\frac{1}{2r} } \autopar{\partial_r - \frac{1}{2r}}$, we have
\begin{multline*}
     \left<-\autopar{\partial^2_r + \frac{1}{4r^2}}u, \autopar{\frac{h(m+1)}{r} - \partial_r \phi}^2 u\right>  = \left\|\autopar{\frac{h(m+1)}{r} - \partial_r \phi} \autopar{\partial_r - \frac{1}{2r}} u\right\|^2 \\
    + \left< \autopar{\partial_r + \frac{1}{2r}} u, \lc \autopar{\partial_r + \frac{1}{2r}}, \autopar{\frac{h(m+1)}{r} - \partial_r \phi}^2  \rc u \right>  \, .
\end{multline*}

Just as the beginning of Point \ref{preuve_propo8_pt1}, we estimate the commutator. We find
$$
\lc \autopar{\partial_r + \frac{1}{2r}}, \autopar{\frac{h(m+1)}{r} - \partial_r \phi}^2  \rc =  e_r \autopar{\partial_r \autopar{\frac{h(m+1)}{r} - \partial_r \phi}^2} .
$$
Thus
\begin{multline*}
\left<\autopar{\partial_r - \frac{1}{2r}}u, \lc \autopar{\partial_r + \frac{1}{2r}}, \autopar{\frac{h(m+1)}{r} - \partial_r \phi}^2  \rc u \right> =  \\
  - \left<u, \partial_r \autopar{\frac{h(m+1)}{r} - \partial_r \phi}^2 \partial_r u \right> - \left<u, \autopar{\partial_r^2  + \frac{\partial_r }{2r}}\autopar{\frac{h(m+1)}{r} - \partial_r \phi}^2 u \right> .
\end{multline*}
By integration by parts,
\begin{equation*}
    2 \mathop{\mathrm{Re}} \left<u, \partial_r \autopar{\frac{h(m+1)}{r} - \partial_r \phi}^2 \partial_r u \right>  = - \left<u, \partial_r^2 \autopar{\frac{h(m+1)}{r} - \partial_r \phi}^2  u \right>  \,. 
\end{equation*}
Therefore,
\begin{multline*}
2 \mathop{\mathrm{Re}} \left< \autopar{\partial_r + \frac{1}{2r}} u, \lc \autopar{\partial_r + \frac{1}{2r}}, \autopar{\frac{h(m+1)}{r} - \partial_r \phi}^2  \rc u \right> = \\ \int |u|^2 \left( -\Delta_r \right)  \autopar{\frac{h(m+1)}{r} - \partial_r \phi}^2  \diff r \,.
\end{multline*}
\item Let us prove the inequality (\ref{propo4.inegalite}).
\begin{enumerate}[label = (\roman*)]
\item \label{preuve_propo8_pt2_i} Let us start by offsetting the two negative terms identified in the previous point.\\
From Lemma \ref{lemme_ellipticite_spectre_ss_ensemble},
$$
\left\| d^\times_{h,m} d_{h,m} u\right\| \geq 2 h B_0 \left\|u\right\| \, ,
$$
then for sufficiently small $h$,
\begin{equation} \label{terme_pathologique1}
\| d^\times_{h,m} d_{h,m} u \|^2 + h^3 \int |u|^2 \left( -\Delta_r B \right) \diff r  \geq h^2 \left( C^2 - h \sup \left( \Delta_r B \right) \right) \| u \|^2 \geq 0 \, .
\end{equation}
For the second term, note that
$$
\| d^\times_{h,m} d_{h,m} u \|^2 \geq 2 h^2 \left\| \autopar{\frac{h (m+1)}{r} - \partial_r \phi} u \right\|^2 \, ,
$$
thus
\begin{equation} \label{terme_pathologique2_inter}
\| d^\times_{h,m} d_{h,m} u \| \geq C h  \left\| \frac{h (m+1)}{r} u \right\| \, .
\end{equation}
Then 
\begin{multline*}
\left( -\Delta_r \right)  \autopar{\frac{h (m+1)}{r} - \partial_r \phi}^2  = -\frac{h^2 (m+1)^2}{r^2}\autopar{1+ \frac{4}{r^2}} +  
\\
\autopar{\frac{h(m+1)}{r} + \frac{\partial_r \phi}{r^2} - \frac{\partial_r^2 \phi}{r} + \partial^3_r \phi }^2 - R_\phi \, ,
\end{multline*}
with $R_\phi = \autopar{\frac{\partial_r \phi}{r^2} - \frac{\partial_r^2 \phi}{r} + \partial^3_r \phi }^2 + \autopar{\partial_r^2 \phi}^2 + \partial_r \phi \autopar{\partial_r^2 \phi + \partial_r^3 \phi}$. \\
Therefore, there exists a constant $D>0$ independent of $h$ and $m$ such that
\begin{equation} \label{terme_pathologique2}
h^2 \int |u|^2 \left( -\Delta_r \right) \autopar{\frac{h (m+1)}{r} - \partial_r \phi}^2 \diff r \geq -Dh^2 \left( \left\| \frac{h (m+1)}{r} u \right\|^2 + \| u \|^2   \right) \, .
\end{equation}
Finally, using points (\ref{terme_pathologique2_inter}), (\ref{terme_pathologique2}) and Lemma \ref{lemme_ellipticite_spectre_ss_ensemble}, we have the existence of $C>0$ independent of $h$ and $m$ such that for $h$ sufficiently small
\begin{multline*}
\autopar{\frac{D}{C^2} +1}\| d^\times_{h,m} d_{h,m} u \|^2 + h^2 \int |u|^2 \left( -\Delta_r \right) \autopar{\frac{h (m+1)}{r} - \partial_r \phi}^2 \diff r \\ 
\geq \sqrt{h} \left( C^2 - D h \right) \| u \|^2 \geq 0 \, .
\end{multline*}

\item Point \ref{preuve_propo8_pt2_i} ensures that, modulo the modification of the constant in front of $\| d^\times_{h,m} d_{h,m} u \|$, each of the terms computed in point \ref{preuve_propo8_pt1} are positive. We can easily conclude that there exist $C>0$ independent of $h$ and $m$ such that for $h$ sufficiently small 
\begin{multline*}
C\|d^\times_{h,m}d_{h,m}u\| \geq \sqrt{h} \left( \|u\| + \left\| \left(\frac{h(m+1)}{r}-\partial_r \phi\right)u\right\| \right) 
+h^2
\left\|\left(\partial_r^2+\frac{1}{4r^2}\right)u\right\|
\\
+ h \left\|
\left(\frac{h(m+1)}{r}-\partial_r\phi\right)\left(\partial_r -\frac{1}{2r}\right)u
\right\|
+ h^{3/2} \left\|\left(\partial_r -\frac{1}{2r}\right)u\right\| \,.
\end{multline*}
\end{enumerate}
\end{enumerate}
    \end{proof}
    \begin{lemma} \label{lemme_inegalite_norme_bord}
Let $h>0$ and $m\in\mathbb{R}$. We have for all $u\in H^1_0\cap H^2 \left( \cI_\delta , \CC \right)$, 
\[\begin{split}
&{|d_{h,m}u|^2\partial_n\phi}_{|r = \rho_2} + {|d_{h,m}u|^2\partial_n\phi}_{|r = \rho_1}
\\&=
\int|d_{h,m}u|^2 \left(\partial_r^2\phi-r^{-1}\partial_r\phi\right) \diff r
+\int 2\Re \left<d_{h,m}u,\partial_r\phi\left(\partial_r+\frac{1}{2r}\right)d_{h,m}u \right> \diff r \,,
\end{split}\]
\[
\left(\partial_r+\frac{1}{2r}\right)d_{h,m}
=
-ih\left(\partial_r^2+\frac{1}{4r^2}\right)
+
\left(\frac{h(m+1)}{r}-\partial_r\phi\right)\left(\partial_r-\frac{1}{2r}\right)
-B\,.
\]
Moverover, there exist $h_0 > 0$ and a constant $C>0$ (independent of $h$ and $m$) such that for all $h\in (0,h_0)$ and $m\in \RR$
\begin{equation} \label{eq.propo.ineg.bord}
    {|d_{h,m}u|^2\partial_n\phi}_{|r = \rho_2} + {|d_{h,m}u|^2\partial_n\phi}_{|r = \rho_1}
\leq 
Ch^{-3/2}\|d_{h,m}^\times d_{h,m}u\|^2\,.
\end{equation}

    \end{lemma}
\begin{proof}
Let $h\in]0,1]$, $m\in\mathbb{R}$ and $u\in H^1_0\cap H^2 \left( \cI_\delta , \CC \right)$.
\begin{enumerate}[label = (\roman*)]
\item \label{lemme5.4_pt.i} Concerning the first identity, it is enough to notice that
\begin{equation*} 
\begin{split}
 \int \autopar{\partial_r + \frac{1}{r}} \autopar{ \partial_r \phi \valabs{d_{h,m}u}^2 } &\diff r = \int \autopar{\lc \autopar{\partial_r + \frac{1}{r}} , \partial_r \phi \rc + \partial_r \phi \autopar{\partial_r  + \frac{1}{r}}}  |d_{h,m} u|^2 \diff r \\
&= \int |d_{h,m} u|^2 \partial_r^2 \phi + 2 \mathop{\mathrm{Re}} \left< d_{h,m} u , \partial_r \phi \autopar{\partial_r + \frac{1}{2r}} d_{h,m} u \right> \diff r ,
\end{split}
\end{equation*}
and 
\begin{multline*}
\int \autopar{\partial_r + \frac{1}{r}} \autopar{ \partial_r \phi \valabs{d_{h,m}u}^2 } \diff r= {|d_{h,m}u|^2\partial_n\phi}_{|r = \rho_2} + {|d_{h,m}u|^2\partial_n\phi}_{|r = \rho_1} \\ +  \int \valabs{d_{h,m}u}^2 \frac{\partial_r \phi}{r} \diff r .
\end{multline*}

\item For the second equality, we can rewrite the definition of $d_{h,m}$ as follows
$$
d_{h,m} = -ih \autopar{\partial_r - \frac{1}{2r}} -i \autopar{\frac{h(m+1)}{r} - \partial_r \phi} .
$$
By composing by $\autopar{\partial_r + 1/2r}$, we have
\begin{multline*}
\autopar{\partial_r + \frac{1}{2r}} d_{h,m} u = -ih \autopar{\partial_r^2 + \frac{1}{4r^2}}u -i \lc \autopar{\partial_r + \frac{1}{2r}} , \autopar{\frac{h(m+1)}{r} - \partial_r \phi}  \rc u \\ -i \autopar{\frac{h(m+1)}{r} - \partial_r \phi} \autopar{\partial_r + \frac{1}{2r}}u \,.
\end{multline*}
The terms of the commutator reorganize in the following way
$$
\lc \autopar{\partial_r + \frac{1}{2r}} , \autopar{\frac{h(m+1)}{r} - \partial_r \phi}  \rc u = - \autopar{\frac{h(m+1)}{r} - \partial_r \phi} \frac{u}{r} - Bu\, ,
$$
which proves the result.

\item From Lemmas \ref{lemme.calcul_norme2_d_h,m} and \ref{lemme_inegalite_inter}, we have
\[
\begin{split}
&\| d_{h,m} u \|^2 
\\
& =
h^2\left\|\left(\partial_r-\frac{1}{2r}\right)u\right\|^2
+
\left\|
\left(\frac{h(m+1)}{r}-\partial_r\phi\right)u
\right\|^2
+
h\left\|
\sqrt{B}u
\right\|^2 \leq C h^{-1} \| d^\times_{h,m} d_{h,m} u \|^2 ,
\end{split}
\]
and
\[
\begin{split}
& \left\| \left(\partial_r+\frac{1}{2r}\right)d_{h,m} u \right\| 
\\
& \leq
h \left\| \left(\partial_r^2+\frac{1}{4r^2}\right) u \right\|
+
\left\| \left(\frac{h(m+1)}{r}-\partial_r\phi\right)\left(\partial_r-\frac{1}{2r}\right) u \right\|
+ \|Bu\| \leq C h^{-1} \| d^\times_{h,m} d_{h,m} u \| .
\end{split}
\]
Thus, using the Cauchy-Schwarz inequality on point \ref{lemme5.4_pt.i}, there exists $C>0$ independent of $h$ and $m$ such that
\begin{equation*}
\begin{split}
& \int|d_{h,m}u|^2 \left(\partial_r^2\phi-r^{-1}\partial_r\phi\right) \diff r
+\int 2\Re \left<d_{h,m}u,\partial_r\phi\left(\partial_r+\frac{1}{2r}\right)d_{h,m}u \right> \diff r
\\
& \leq \left( \sup \, (\partial_r^2 \phi - r^{-1} \partial_r \phi) \right) \| d_{h,m} u \|^2  + 2 \| d_{h,m} u \| \left\| \left(\partial_r+\frac{1}{2r}\right)d_{h,m} u \right\| \leq C h^{-3/2} \| d^\times_{h,m} d_{h,m} u \|^2 \, .
\end{split}
\end{equation*}
\end{enumerate}
\end{proof}
\begin{proof}[Proof of Proposition \ref{proposition_inegalite_ellipticite}] \label{proof.inegalite.elliptique} Let $h>0$ and $m\in \RR$, recall that Proposition \ref{propo.Dirac} gives us the identity 
\begin{equation} \label{eq.proofelliptique.imker}
\mathrm{ker} \autopar{d_{h,m}^*}^{\perp} \cap \mathrm{Dom} \autopar{d_{h,m}^*} = \left\{ d_{h,m} u ; \; \; u \in H_0^1 \cap H^2 \autopar{[\rho_1, \rho_2], \CC} \right\}.    
\end{equation}
\begin{enumerate}
    \item First inequality is a direct consequence of \eqref{eq.proofelliptique.imker} and the second inequality of Lemma \ref{lemme_ellipticite_spectre_ss_ensemble}.
    \item For the second one, it is enough to use \eqref{eq.proofelliptique.imker} and \eqref{eq.propo.ineg.bord}. 
\end{enumerate}
\end{proof}
\subsection{Lower bound, for $k\geq 2$ of $\lambda_{k,m} (h)$}
Let us fix $p\in \ZZ$. We will note in this subpart $\widetilde{m} = m - \gamma_{h,p}$, with $h>0$ and $\gamma_{h,p}$ defined in Proposition \ref{propo.potentiel.flux}.

\medskip

The following proposition describes the energy levels of $\cL_{h,m}^{-} = d_{h,\widetilde{m}} d_{h,\widetilde{m}}^\times $ different from the ground state, in the semiclassical limit. 
\begin{proposition} \label{propo_minoration_haut_du_spectre}
Let $h >0$, $k\geq 2$. For all $m\in \RR$, we have 
$$
\lambda_{k,m} (h) \geq 2 h B_0 .
$$
\end{proposition}

\begin{proof}
Let $k\geq 2$, $h >0$ and $m\in \RR$.\\
The min-max theorem allows to order the eigenvalues of the Pauli operator with Dirichlet and Neumann boundary conditions,
\begin{equation*}
    \lambda_{k,m} (h) = \underset{\underset{\mathrm{dim}V = k}{V \subset H^{1}_{0}}}{\mathrm{inf}} \; \underset{u\in V\backslash \{ 0 \} }{\mathrm{sup}} \; \frac{\norm{ d^\times_{h,m} u }^2}{\norm{u}^2 } \; \geq  \underset{\underset{\mathrm{dim}V = k}{V \subset H^{1}}}{\mathrm{inf}} \; \underset{u\in V\backslash \{ 0 \} }{\mathrm{sup}} \; \frac{\norm{d^\times_{h,m} u }^2}{\norm{u}^2} \; = \nu_{k,m} (h),
\end{equation*}
with $\nu_{k,m} (h)$ the $k$-th eigenvalue of $\cL^{N}_{h,m}$, the operator $\cL_{h,m}^-$ with homogeneous Neumann boundary conditions. \\
Observe that the first eigenvalue of the Pauli operator with Neumann conditions is zero, indeed, the function $r \mapsto r^{m +1/2}$ cancels the quadratic form. This implies that for $k \geq 2$,
$$
\mathrm{ker}\autopar{\cL^{N}_{h,m} - \nu_{k,m} (h)} \subset \mathrm{ker}\autopar{\cL^{N}_{h,m} }^{\perp} =  \mathrm{ker} \autopar{\diff_{h,m}^\times}^{\perp}.
$$
However, according to Proposition \ref{proposition_inegalite_ellipticite}, for all $u \in \mathrm{Dom} \autopar{\diff_{h,m}} \cap \mathrm{ker} \autopar{\diff_{h,m}^\times}^{\perp}$,
$$
\left\|\diff_{h,m}^\times u\right\|^2 \geq  2 h B_0  \left\| u\right\|^2 .
$$
Thus, for $k \geq 2$ and $u\in \mathrm{ker}\autopar{\cL^{N}_{h,m} - \nu_{k,m} (h)}$
$$
\nu_k (h) = \frac{\norm{d^\times_{h,m} u }^2}{\norm{u}^2} \; \geq  2 h B_0 .
$$
This completes the proof.
\end{proof}

We directly deduce from Proposition \ref{propo_minoration_haut_du_spectre} and Proposition \ref{propo.CVU.coercivite} (\ref{pt3.propCVU_coercivite}).

\begin{corollary} \label{corollaire_demo_pt3_propo}
For all $h>0$, $j\geq 2$ and $m\in \RR$, we have
\[
f_{j,h}(m)\geq 2B_0 \sqrt{h}e^{-2\phi{\rm min}/h}\,.
\]
\end{corollary}
\section{Upper bound and consequences} 
\label{sec:majoration}
In this section we establish an upper bound for the set of ground states of the operators $\cL_{h,m}^-$. This result is a consequence of Lemmas \ref{lem:majoration_numerateur} and \ref{lem:minoration_denominateur}. We then show in Section \ref{subsubsection:localisation} that, from the upper bound, the eigenfunctions are localized in the neighborhood of $\phi$ minimum. Then, by combining the upper bound and the elliptic estimates of Section \ref{section.estim.elliptiques} we get a monomial approximation of the eigenfunctions inside the domain, see Proposition \ref{propo:approx_holo}. 
\subsection{Upper bound} \label{subsection_borne_sup}
In this subsection, we seek to obtain for any $m \in \RR$ an upper bound for $\lambda_{1,m} (h)$. More precisely we will establish the following proposition. 
\begin{proposition} \label{propo.borne_sup_lambda1m}
Let $\alpha \in ]1/2, 1[$ and $\beta \in ]1/3, 1/2[$. There exists $h_0 > 0$ such that for all $h \in ]0, h_0]$ and for all $m \in \RR$, we have
$$
\lambda_{1,m} (h) \leq C_{m} (h) \; \sqrt{h} \; e^{\frac{2 \phi_{\min}}{h}} \; (1+o_{h \rightarrow 0}(1)).
$$
with 
\begin{multline*}
C_m (h) = 2 \sqrt{\frac{\phi''_{\min}}{\pi}} \left( \partial_n \phi \autopar{\rho_1} \autopar{ \frac{ \rho_1 + h^\alpha \mathds{1}_{m \geq 0}  }{ r_{\min} - \mathrm{sgn}\autopar{2m+1} h^\beta }  }^{2m+1}\right. \\
+ \left. \partial_n \phi \autopar{\rho_2} \autopar{ \frac{ \rho_2 - h^\alpha \mathds{1}_{m< 0}   }{ r_{\min} - \mathrm{sgn}\autopar{2m+1} h^\beta }  }^{2m+1} \right) .
\end{multline*}
\end{proposition}
To do this, we look for a suitable test function to reinject into the Rayleigh quotient given in Proposition \ref{corollaire.minmax.mom_reel}. Heuristically, a minimizer of \eqref{eq.minmax.moment.reels} wants to be in the kernel of $\autopar{ \partial_{r} - (m + 1/2)/r}$ in the interior of $\Omega$. \\ 
Therefore, we choose to take test functions in $\left\{ a\, r^{m+1/2}, \; \text{ with } a \in \CC  \right\}$ up to the multiplication by a sufficiently smooth cutoff function for Dirichlet boundary condition.

\medskip

Consider the following cutoff function
\begin{definition} \label{definition:fonction_test}
Let $\varepsilon>0 $, define for all $r\in [\rho_1 , \rho_2]$ 
$$
\chi (r) =
      \left\{
      \begin{array}{ll}
        P \autopar{r} & \text{if} \; \max\, (r-\rho_1, \rho_2-r) \leq \varepsilon,\\
        1  & \text{otherwise.}
     \end{array}
    \right.
$$
with $P$ a sufficiently regular function such that $r \mapsto \chi(r) \, r^{m+1/2}\in H_{0}^1 $.
\end{definition}
Let us start by estimate the energy by putting in a test function like $r \mapsto \chi(r) \, r^{m+1/2}$.
\begin{lemma} \label{lem:majoration_numerateur}
Let $\alpha \in ]1/2, 1[$, there exists $h_0 > 0$ such that for all $h \in ]0, h_0]$ and for all $m \in \RR$, we have
\begin{multline*}
h^2 \int_{\rho_{1}}^{\rho_{2}}  e^{-2\phi / h} \valabs{\autopar{ \partial_{r} - \frac{m+1/2}{r}} \left( \chi(r) \, r^{m+1/2} \right)}^2  \diff r \\
\leq 2 h \autopar{\partial_n\phi \autopar{\rho_1}  \autopar{\rho_1 + h^\alpha \mathds{1}_{m \geq 0} }^{2m+1}  + \partial_n\phi\autopar{\rho_2} \autopar{\rho_2 - h^\alpha \mathds{1}_{m < 0}  }^{2m+1} } \autopar{1+ o(1) }
\end{multline*}
with $\chi$ defined in Definition \ref{definition:fonction_test}.
\end{lemma}
\begin{proof}Let $m \in \RR$ and $h, \varepsilon > 0$, we have
\begin{multline*}
h^2 \int_{\rho_{1}}^{\rho_{2}}  e^{-2\phi / h} \valabs{\autopar{ \partial_{r} - \frac{m+1/2}{r}} \left( \chi(r) \, r^{m+1/2} \right)}^2  \diff r \\
= h^2 \lc \underset{(a)}{\int_{\rho_{1}}^{\rho_{1} + \varepsilon}} + \underset{(b)}{\int^{\rho_{2}}_{\rho_{2} - \varepsilon}}  \rc e^{-2\phi / h}  r^{2m+1}   \valabs{ \partial_r P(r)}^2 \mathrm{d}r.
\end{multline*}
Let's start by estimate the integral $(a)$; using the Taylor expansion of $\phi$ in a neighborhood of the interior boundary,
$$
\phi(r) = - \autopar{r- \rho_1 } \partial_n \phi \autopar{\rho_1 } + \cO \autopar{\varepsilon^2} \; \; \text{with } \rho_1 \leq r \leq \rho_1 + \varepsilon,
$$
thus
\begin{align*}
\int_{\rho_{1}}^{\rho_{1} + \varepsilon} e^{-2\phi / h}  r^{2m+1}  & \valabs{ \partial_r P(r)}^2 \mathrm{d}r =  \int_{\rho_{1}}^{\rho_{1} + \varepsilon} e^{\frac{2}{h} (r-\rho) \partial_n \phi\autopar{\rho_{1}}}  r^{2m+1} \valabs{\partial_r P(r)}^2  \mathrm{d}r \; \autopar{1+\cO \autopar{\frac{\varepsilon^2}{h}}} \\
& \leq   \autopar{\rho_1 + h^\alpha \mathds{1}_{m \geq 0} }^{2m+1}  \int_{0}^{\varepsilon} e^{\frac{2}{h} \tau \partial_n \phi\autopar{\rho_{1}}}  \valabs{\partial_{\tau} \widetilde{P}(\tau)}^2  \mathrm{d}\tau \autopar{1 +\cO \autopar{ \frac{\varepsilon^2}{h}}}.
\end{align*}
We can now choose the optimal $P$ using \cite[Lemma A.1.]{barbaroux:hal-01889492} and obtain
\begin{align*}
    \int_{0}^{\varepsilon} e^{\frac{2}{h} \tau \partial_n \phi\autopar{\rho_{1}}}  \valabs{\partial_{\tau} \widetilde{P}(\tau)}^2  \mathrm{d}\tau & \leq \frac{2 \partial_n \phi \autopar{\rho_{1}} /h }{1- \mathrm{exp}\autopar{-\frac{2}{h} \varepsilon \partial_n \phi \autopar{\rho_{1}} }} \\
    & = \frac{2 \partial_n \phi \autopar{\rho_1}}{h} \autopar{1 + \cO \autopar{e^{-\frac{2 \varepsilon}{h} \partial_n \phi \autopar{\rho_{1}}}}}.
\end{align*}

For the integral $(b)$, the estimate follows the same pattern. In a neighborhood of the exterior boundary, the Taylor expansion of $\phi$ write
$$
\phi(\tau) = - \autopar{\rho_{2} - \tau} \partial_n \phi \autopar{\rho_{2}} + \cO \autopar{\varepsilon^2}\; \; \text{with } \rho_2 - \varepsilon \leq \tau \leq \rho_2,
$$
thus
\begin{align*}
\int^{\rho_{2}}_{\rho_{2} - \varepsilon} e^{-2\phi / h}  r^{2m+1} &  \valabs{ \partial_r P(r)}^2 \mathrm{d}r  =  \int_{\rho_{2}-\varepsilon}^{\rho_{2}} e^{\frac{2}{h}\autopar{\rho_{2} - \tau} \partial_{n} \phi\autopar{\rho_{2}}} \tau^{2m+1} \valabs{\partial_{\tau} P(\tau)}^2 \diff \tau \autopar{1+ \cO \autopar{\frac{\varepsilon^2}{h}}} \\
& \leq \autopar{\rho_2 - h^\alpha \mathds{1}_{m < 0}  }^{2m+1} \int^{\varepsilon}_0 e^{\frac{2}{h}\tau \partial_{n} \phi\autopar{\rho_{2}}}  \valabs{\partial_{\tau} \widetilde{P}(\tau)}^2 \diff \tau \autopar{1+ \cO \autopar{\frac{\varepsilon^2}{h}}}.
\end{align*}
In the same way as for the integral (a), we choose the optimal $P$, hence
\begin{align*}
    \int_{0}^{\varepsilon} e^{\frac{2}{h} \tau \partial_n \phi\autopar{\rho_{2}}}  \valabs{\partial_{\tau} \widetilde{P}(\tau)}^2  \mathrm{d}\tau \leq \frac{2 \partial_n \phi \autopar{\rho_2}}{h} \autopar{1 + \cO \autopar{e^{-\frac{2 \varepsilon}{h} \partial_n \phi \autopar{\rho_{2}}}}}.
\end{align*}
To control the remainders, let us take $\varepsilon = h^\alpha$ with $\alpha \in ]1/2, 1[$. Combining the upper bounds of $(a)$ and $(b)$ we get the result.
\end{proof}
Now we give a lower bound for the $L^2$-norm using the same test function.
\begin{lemma} \label{lem:minoration_denominateur}
Let $\beta \in ]1/3, 1/2[$, there exists $h_0 >0$ such that for all $ h\in ]0,h_0[$ and for all $ m \in \RR$ we have
$$
\left\|e^{-\phi/h} \chi \, r^{m+1/2}\right\|^2 \geq  \sqrt{\frac{\pi}{\phi_{\min}''}} \; \autopar{r_{\min} - \mathrm{sgn}\autopar{2m+1} h^\beta }^{2m+1} \; \sqrt{h} \;  e^{-2\phi_{\min} / h}(1+o(1)) \; ,
$$
with $\phi_{\min}''$ defined in Remark \ref{remark.phimin}.
\end{lemma}
\begin{proof}
Let $m \in \RR$ and $h >0$. We choose the same test function in the previous proof. \\
There exists $\eta_0 >0$ such that for all $\eta \in ]0,\eta_0]$, the annulus of center (0,0) and of radius $r_{\min} - \eta $ and $r_{\min} + \eta$ is included in $\Omega$ and $\phi$ admits the Taylor expansion,
$$
\phi ( r) = \phi_{\min} + \frac{\autopar{r-r_{\min}}^2}{2} \phi''_{\min} + \cO \autopar{\eta^3} \; \; \text{with } \; r_{\min} - \eta \leq r \leq r_{\min} + \eta.
$$
Therefore, we have 
\begin{align*}
    \left\|e^{-\phi/h} \chi \, r^{m+1/2}\right\|^2   \geq \int_{r_{\min}-\eta}^{r_{\min}+\eta} e^{-2\phi/h} r^{2m+1} \diff r 
\end{align*}
and 
\begin{align*}
    & \int_{r_{\min}-\eta}^{r_{\min}+\eta}  e^{-2\phi/h} r^{2m+1} \diff r \geq\\
    &   \autopar{r_{\min} - \mathrm{sgn}\autopar{2m+1} \eta }^{2m+1} e^{-2\phi_{\min} / h} \int_{r_{\min}-\eta}^{r_{\min}+\eta} e^{ - \autopar{r-r_{\min}}^2 \phi''_{\min}/h } \diff r \autopar{1 + \cO \autopar{\frac{\eta^3}{h}}} \; .
\end{align*}
It is now enough to consider the change of variable $\tau =\autopar{r-r_{\min}}\sqrt{\phi''_{\min}/h}$ to get
\begin{align*}
    \int_{r_{\min}-\eta}^{r_{\min}+\eta} \mathrm{exp}\autopar{ - \frac{\autopar{r-r_{\min}}^2}{h} \phi''_{\min}} \diff r &= \sqrt{h} \; \int_{-\frac{\eta}{\sqrt{h}}}^{\frac{\eta}{\sqrt{h}}} \mathrm{exp} \autopar{- \phi''_{\min} \tau^2 } \diff \tau \\
    & = \sqrt{h} \; \int_\RR \mathrm{exp} \autopar{- \phi''_{\min} \tau^2}  \diff \tau \autopar{1 + o_{ \frac{\eta}{\sqrt{h}} \rightarrow +\infty} (1)}\\
    & = \sqrt{\frac{\pi h}{\phi''_{\min}}} \autopar{1 + o_{\frac{\eta}{\sqrt{h}} \rightarrow +\infty} (1)} \; .
\end{align*}

Taking, $\eta = h^\beta$ with $\beta \in ]1/3, 1/2[$, we have
$$
  \left\|e^{-\phi/h} \chi \, r^{m+1/2}\right\|^2  \geq  \sqrt{\frac{\pi}{\phi_{\min}''}} \; \autopar{r_{\min} - \mathrm{sgn}\autopar{2m+1} h^\beta }^{2m+1} \; \sqrt{h} \;  e^{-2\phi_{\min} / h}(1+o(1)) \; .
$$
\end{proof}
By combining the two lemmas we prove Proposition \ref{propo.borne_sup_lambda1m}.
\subsection{Approximation results} \label{subsection_consequences_borne_sup}
We focus in this section on the implications of the upper bound of Proposition \ref{propo.borne_sup_lambda1m} and the elliptic estimates of Section \ref{section.estim.elliptiques}. From these results we deduce the localization in the neighborhood of the minimum of $\phi$ and the monomial approximation of the eigenfunctions associated to $ \lambda_{1,m} (h)$.
\subsubsection{Localization of the eigenfunctions} \label{subsubsection:localisation}
The upper bound implies that the eigenfunctions are localized in the neighborhood of the minimum of $\phi$, i.e. in the neighborhood of $r_{\min}$. Let us start with a technical lemma which will be useful in the proof of Proposition \ref{propo:localisation_fct_propre}.
\begin{lemma} \label{lemme.localisation.fonction_propre}
For all $ h >0$ and $ m \in \RR$, and for all $v\in H_0^1$, we have
\begin{equation*}
h^2 \int_{\rho_{1}}^{\rho_{2}}  e^{-2\phi / h} \valabs{\autopar{ \partial_{r} - \frac{m+1/2}{r}} v(r)}^2  \geq h^2 a_0 \int_{\rho_{1}}^{\rho_2} |v|^2 \, ,
\end{equation*}
where $a_0 >0$ is the smallest eigenvalue of the Dirichlet Laplacian on $\Omega$.
\end{lemma}
\begin{proof}
Let $v \in H_0^1 $. Using the fact that $\phi \leq 0$, it follows,
\begin{equation*}
h^2 \int_{\rho_{1}}^{\rho_{2}}  e^{-2\phi / h} \valabs{\autopar{ \partial_{r} - \frac{m+1/2}{r}} v_{h,m}(r)}^2  \diff r \geq h^2 \int_{\rho_{1}}^{\rho_{2}}  \valabs{\autopar{ \partial_{r} - \frac{m+1/2}{r}} v_{h,m}(r)}^2  \diff r .
\end{equation*}
By integration by parts,
\begin{equation*}
\int_{\rho_{1}}^{\rho_{2}} \valabs{\autopar{ \partial_{r} - \frac{m+1/2}{r}} v}^2   =  \int_{\rho_{1}}^{\rho_{2}} \left| \left( \partial_r - \frac{1}{2} \right) v \right|^2 + \left|  \frac{m}{r}v \right|^2  .
\end{equation*}
Finally we have the lower bound
\begin{equation*}
\int_{\rho_{1}}^{\rho_{2}} \left| \left( \partial_r - \frac{1}{2} \right) v \right|^2 \geq  a_0 \int_{\rho_{1}}^{\rho_2} |v|^2 \, ,
\end{equation*}
$a_0 >0$ being the smallest eigenvalue of the Dirichlet Laplacian on $\Omega$.
\end{proof}
We have the following localization property.
\begin{proposition} \label{propo:localisation_fct_propre}
Let $\alpha \in ]0,1/2[ $, $ h \in ]0, h_0]$, $ m \in \RR$ and $v_{h,m}$ an eigenfunction associated to $\lambda_{1,m} (h)$. We have
\begin{equation*}
\frac{\int_{r_{\min} - h^{\alpha}}^{r_{\min} + h^{\alpha}} e^{-2 \phi/h}  \vert v_{h,m} \vert^2   }{\int_{\rho_{1}}^{\rho_{2}} e^{-2 \phi/h}  \vert v_{h,m}  \vert^2} = 1 + f_{1,m} (h) \cO(h^{\infty}) \, ,
\end{equation*}
and for $\delta >0$ sufficiently small
\begin{equation*}
\frac{\int_{\rho_{1} + \delta}^{\rho_{2} - \delta} e^{-2 \phi/h}  \vert v_{h,m} \vert^2   }{\int_{\rho_{1}}^{\rho_{2}} e^{-2 \phi/h}  \vert v_{h,m}  \vert^2} = 1 + f_{1,m} (h) \cO(h^{\infty}) \, ,
\end{equation*}
with $h_0$ and $C_m (h)$ defined in Proposition \ref{propo.borne_sup_lambda1m} and $\cO \autopar{h^\infty}$ independent of $m$.
\end{proposition}
\begin{proof}
Firstly, let us note that
$$   
\frac{\int_{r_{\min} - h^{\alpha}}^{r_{\min} + h^{\alpha}} e^{-2 \phi/h}  \vert v_{h,m} \vert^2   }{\int_{\rho_{1}}^{\rho_{2}} e^{-2 \phi/h}  \vert v_{h,m}  \vert^2} = 1- \frac{\autopar{\int_{\rho_{1}}^{r_{\min} - h^{\alpha}} + \int_{r_{\min} + h^{\alpha}}^{\rho_{2}} } e^{-2 \phi/h}  \vert v_{h,m}  \vert^2 }{\int_{\rho_{1}}^{\rho_{2}} e^{-2 \phi/h}  \vert v_{h,m}  \vert^2} .
$$   
Furthermore, since $\phi$ is a solution of the Poisson's equation \eqref{eq.phi} and $B>0$, the maximum principle ensures that 
\begin{align*}
    \underset{r\in [\rho_1 , \rho_2] \backslash [r_{\min}-h^{\alpha},r_{\min}+h^{\alpha}]}{inf} \phi(r) & = \min \{ \phi(r_{\min} - h^{\alpha}), \phi (r_{\min}+h^{\alpha}) \} \\
    & \geq \phi_{\min} + \frac{\phi_{\min}^{''}}{2} h^{2\alpha} \autopar{1+ \cO \autopar{h^{\alpha}}},
\end{align*}

then
\begin{equation*}
    \autopar{\int_{\rho_{1}}^{r_{\min} - h^{\alpha}} + \int_{r_{\min} + h^{\alpha}}^{\rho_{2}} } e^{-2 \phi/h}  \vert v_{h,m}  \vert^2 \leq e^{-2\phi_{\min}/h - 2\phi''_{\min} h^{2\alpha -1} \autopar{1+ \cO \autopar{h^{\alpha}}}} \int_{\rho_{1}}^{\rho_{2}}  \vert v_{h,m}  \vert^2  .
\end{equation*}
Moreover, according to Proposition \ref{propo.borne_sup_lambda1m}, 
\begin{equation*}
\int_{\rho_{1}}^{\rho_2} |v_{h,m}|^2 \leq a_0^{-1}  \int_{\rho_{1}}^{\rho_{2}}  e^{-2\phi / h} \valabs{\autopar{ \partial_{r} - \frac{m+1/2}{r}} v(r)}^2   \, ,
\end{equation*}
and
\begin{equation*}
h^2 \int_{\rho_{1}}^{\rho_{2}}  e^{-2\phi / h} \valabs{\autopar{ \partial_{r} - \frac{m+1/2}{r}} v(r)}^2  = f_{1,h} (m) \; h^{1/2} \; e^{\frac{2 \phi_{\min}}{h}} \; \int_{\rho_{1}}^{\rho_2} e^{-2\phi/h}|v_{h,m}|^2  .
\end{equation*}
Thus
\begin{multline*}
    \autopar{\int_{\rho_{1}}^{r_{\min} - h^{\alpha}} + \int_{r_{\min} + h^{\alpha}}^{\rho_{2}} } e^{-2 \phi/h}  \vert v_{h,m}  \vert^2 \leq \\
    f_{1,h} (m) \;  a_0^{-1} h^{-3/2} e^{-2\phi_{\min}''/h} \int_{\rho_{1}}^{\rho_{2}} e^{-2 \phi/h}  \vert v_{h,m}  \vert^2   (1+o_{h \rightarrow 0}(1)),
\end{multline*}
and the conclusion follows.
\end{proof}
\subsubsection{Monomial approximation} \label{subsubsection_approx_holo}
Following the results of Proposition \ref{propo.Dirac}, we define the orthogonal projection on $\mathrm{ker} \autopar{d_{h,m}^\times}$.\\ Remember that
$$
\mathrm{ker} \autopar{d_{h,m}^\times} = \mathrm{Vect} \autopar{r \mapsto e^{- \phi /h} r^{m +1/2}}\, .
$$
\begin{definition} \label{def_projection_holo}
Let $\delta, h >0$ and $m \in \RR$. We define $\Pi_{h,\delta}$ the orthogonal projection on the kernel of $d_{h,m}^\times$ defined on the domain $H_0^1 \left( \cI_{\delta} , \CC \right)$ with $\cI_{\delta} := ]\rho_1 + \delta, \rho_2 - \delta[$. \\
Moreover, if $u= e^{-\phi/h}v$, we write $\Pi_{h,\delta} u = e^{-\phi/h} \widetilde{\Pi}_{h,\delta} v. $
\end{definition}
\begin{notation} \label{norme_l2_poids} Recall that $\partial_n \phi (\rho_1) , \partial_n \phi (\rho_2) >0$, see Subsection \ref{subsubsection.phi}. The following norm will be used hereafter, for all $X = (x_1, x_2) \in \CC^2$ 
$$
\cN_{\partial_{n} \phi} (X) = \sqrt{ \partial_n \phi (\rho_1) \valabs{x_1 }^2 + \partial_n \phi (\rho_2) \valabs{x_2 }^2   } \, .
$$
\end{notation}
From Proposition \ref{propo:localisation_fct_propre} we deduce the following monomial approximation result.
\begin{proposition}
\label{propo:approx_holo}
There exist $C, h_0, \delta_0 >0$ such that for all $\delta \in ]0,\delta_0]$ and $h \in ]0 , h_0[$, we have for all eigenfunction $v_{h,m}$ associated to $\lambda_{1,m} (h)$, 
\begin{equation*}
\norm{e^{- \phi/h} \left( \Id - \widetilde{\Pi}_{h,\delta}  \right)v_{h,m}  }_{L^{2} \autopar{\cI_{\delta}}} \leq C h^{-1/2} \lambda_{1,m} (h) \autopar{ 1+ f_{1,h} (m)} \left\|e^{-\phi/h} v_{h,m}\right\|_{L^{2} \autopar{\cI_{\delta}}} \; ,
\end{equation*}
and
\begin{equation*}
\cN_{\partial_{n} \phi} \left( U_{h,m} \right)  \leq C h^{-3/4} \sqrt{\lambda_{1,m} (h)} \autopar{ 1+ f_{1,h} (m) } \left\|e^{-\phi/h} v_{h,m}\right\|_{L^{2} \autopar{\cI_{\delta}}} \; ,
\end{equation*}    
with $\cN_{\partial_{n} \phi} (\cdot)$ defined in Notation \ref{norme_l2_poids} and 
\begin{equation*}
U_{h,m} = \left( e^{-2\phi(\rho_{1}+ \delta)/h}\left( \Id - \widetilde{\Pi}_{h,m}  \right)v_{h,m} \autopar{\rho_1 + \delta}, e^{-2\phi(\rho_{2}- \delta)/h}\left( \Id - \widetilde{\Pi}_{h,m}  \right)v_{h,m} \autopar{\rho_2 - \delta} \right) .
\end{equation*}
\end{proposition}
\begin{proof}
Let $h>0$, $0 < \delta < (\rho_2 - \rho_1) /2$ and $m \in \RR$. Let $v_{h,m}$ an eigenfunction associated to $\lambda_{1,m} (h)$. \\
Using Proposition \ref{propo:localisation_fct_propre}, there exist $h_0 , \delta_0 >0$ such that for all $h \in ]0 , h_0]$ and $\delta \in ]0, \delta_0 ]$, we have
\begin{align*}
\left\|e^{-\phi / h} h \autopar{\partial_r + \frac{m + 1/2}{r} } v_{h,m} \right\|_{L^{2} (\cI_{\delta})}^2 & \leq \left\|e^{-\phi / h} h \autopar{\partial_r + \frac{m + 1/2}{r} } v_{h,m} \right\|_{L^{2} (\rho_{1} , \rho_{2})}^2  \\
& = \lambda_{1,m} (h) \left\|e^{-\phi/h} v_{h,m}\right\|_{L^{2} (\rho_{1} , \rho_{2})}^2 \\
&\leq \lambda_{1,m} (h) \autopar{ 1+ f_{1,h} (m) \cO \autopar{ h^\infty}} \left\|e^{-\phi/h} v_{h,m}\right\|_{L^{2} \autopar{\cI_{\delta}}}^2 \\
&\leq \lambda_{1,m} (h) \autopar{ 1+ f_{1,h} (m) } \left\|e^{-\phi/h} v_{h,m}\right\|_{L^{2} \autopar{\cI_{\delta}}}^2 ,
\end{align*}
with $\cI_\delta = [\rho_1 + \delta , \rho_2 - \delta ]$. \\
Then, if $u_{h,m} = e^{-\phi/h} v_{h,m}$, we have
\begin{align*}
\left\|e^{-\phi / h} h \autopar{\partial_r + \frac{m + 1/2}{r} } v_{h,m} \right\|_{L^{2} (\cI_{\delta})}^2 & = \left\|e^{-\phi / h} h \autopar{\partial_r + \frac{m + 1/2}{r} } \autopar{\Id - \widetilde{\Pi}_{h,m}} v_{h,m} \right\|_{L^{2} (\cI_{\delta})}^2 \\
&= \left\|d_{h,m}^\times \autopar{\Id - \Pi_{h,m}} u_{h,m} \right\|_{L^{2} (\cI_{\delta})}^2 \, .
\end{align*}
We use Proposition \ref{proposition_inegalite_ellipticite} to complete the proof.
\end{proof}
\section{Proof of Proposition \ref{propo.CVU.coercivite}} \label{section.preuve.propo}
In this last section we prove the first two points of Proposition \ref{propo.CVU.coercivite}. We begin with simplify the Rayleigh quotient given in Proposition \ref{corollaire.minmax.mom_reel}. Then, we prove in the next two subsections coercivity and uniform convergence results for $f_{1,h} (m)$.
\subsection{Simplified Rayleigh quotient} \label{sous_section:reduction_rayleigh}
Proposition \ref{propo:approx_holo} states that it is sufficient to know the behavior of the energy near the boundary of $[\rho_1 , \rho_2]$. This leads us to minimize the numerator of Rayleigh quotient as follows.
\begin{lemma} \label{lemme_reduction_num}
There exist $h_0, \delta_0 >0$ such that for all $\delta \in ]0,\delta_0]$ and $h \in ]0 , h_0]$, we have for all $v \in \mathrm{H}_{0}^1 $, 
\begin{multline*}
\left\| e^{-\phi / h} h \autopar{\partial_r + \frac{m + 1/2}{r} } v \right\|^2  \geq h \left( E_{m}^{int} (h,\delta) \, \valabs{v (\rho_1 + \delta )}^2 \right. \\
\left. + E_{m}^{ext} \autopar{h,\delta} \, \valabs{v (\rho_2 - \delta )}^2 \right) \autopar{1+ o(1)} \; ,
\end{multline*}
when $\delta / h \rightarrow + \infty$, $\delta^2 / h \rightarrow 0$ with
\begin{gather*} 
E^{int}_m (h,\delta) = e^{- \delta_{1}} \frac{2 \partial_n \phi \autopar{\rho_1} + (2mh/ \rho_1) \mathds{1}_{]0, +\infty[} (m) }{e^{2m\delta_{1} \mathds{1}_{]0, +\infty[} (m)}- e^{-2\delta_{1} \rho_{1} \partial_{n} \phi \autopar{\rho_{1}} /h }}  \, , \\
E^{ext}_m (h,\delta) = e^{ \delta_{2}} \frac{2 \partial_n \phi \autopar{\rho_2} - (2mh/ \rho_2) \mathds{1}_{]-\infty,0[} (m) }{e^{-2m\delta_{2} \mathds{1}_{]-\infty,0[} (m)}- e^{-2\delta_{2} \rho_{2} \partial_{n} \phi \autopar{\rho_{2}} /h }} \, ,
\end{gather*} 
where $\delta_1 = \mathrm{ln} \autopar{1+\frac{\delta}{\rho_1}}$ and $\delta_2 = - \mathrm{ln} \autopar{1 - \frac{\delta}{\rho_2}}$. Moreover, $o(1)$ is independent of $m$.
\end{lemma}
\begin{proof}
Let $m\in \RR$ and $0<\delta <( \rho_2 - \rho_1 ) /2$.  Let's start by minimizing the energy by its contribution near the boundary
\begin{multline*}
\left\| e^{-\phi / h} h \autopar{\partial_r + \frac{m + 1/2}{r} } v \right\|^2  \geq \\
h^2 \autopar{\int_{\rho_{1}}^{\rho_{1} + \delta} + \int_{\rho_{2} - \delta}^{\rho_{2}}  } \; e^{-2\phi / h} \valabs{ \autopar{ \partial_r \; - \; \frac{m + 1/2}{r} } v  }^2 : = \cQ^{int}_{m,h,\delta} \autopar{v} + \cQ^{ext}_{m,h,\delta} \autopar{v} \;.
\end{multline*}
We use the Taylor expansion of $\phi$ near the boundary
\begin{equation*}
      \left\{
      \begin{aligned}
        &\phi (r ) = - \autopar{r-\rho_1} \partial_n \phi \autopar{\rho_1} + \cO \autopar{\delta^2}\; \; \text{with } \rho_1 \leq r \leq \rho_1 + \delta,\\
        &\phi (r ) = - \autopar{\rho_2 - r } \partial_n \phi \autopar{\rho_2} + \cO \autopar{\delta^2}\; \; \text{with } \rho_2 - \delta \leq r \leq \rho_2 .
     \end{aligned}
    \right.
\end{equation*}
The integral near the interior boundary becomes
\begin{equation*}
    \cQ^{int}_{m,h,\delta} \autopar{v} = h^2 \int_{\rho_{1}}^{\rho_{1} + \delta} e^{2\autopar{r - \rho_{1}} \partial_{n} \phi \autopar{\rho_{1}}/ h} \valabs{ \autopar{ \partial_r \; - \; \frac{m + 1/2}{r} } v(r)  }^2  \diff r \autopar{1+ \cO \autopar{\frac{\delta^2}{h}}} \; .
\end{equation*}    
By writing, $\tau = \mathrm{ln} \autopar{\frac{r}{\rho_1}}$, $\delta_1 = \mathrm{ln} \autopar{1+\frac{\delta}{\rho_1}}$ and $u (\tau)= v\left( \rho_1 e^\tau \right)$, we get 
\begin{align*}
\cQ^{int}_{m,h,\delta} \autopar{v} &= h^2 \int_{0}^{\delta_{1}} e^{2\autopar{e^{\tau} -1} \rho_{1} \partial_{n} \phi \autopar{\rho_{1}} /h} \valabs{\autopar{\partial_\tau - (m+1/2)} u (\tau)}^2 \frac{e^{-\tau}\diff \tau}{\rho_1} \autopar{1+ \cO \autopar{\frac{\delta^2}{h}}} \\
&= h^2 \int_{0}^{\delta_{1}} e^{\frac{2\tau}{h} \rho_{1} \partial_{n} \phi \autopar{\rho_{1}} } \valabs{\autopar{\partial_\tau - (m+ 1/2 )} u (\tau)}^2 \frac{e^{-\tau} \diff \tau}{\rho_1} \autopar{1+ \cO \autopar{\frac{\delta^2}{h}}} \\
&= h^2 \int_{0}^{\delta_{1}} e^{\frac{2\tau}{h} \rho_{1} \partial_{n} \phi \autopar{\rho_{1}} - 2m \autopar{\delta_{1}-\tau} - \delta_1} \valabs{\partial_\tau w_1 \autopar{\tau}}^2 \frac{\diff \tau}{\rho_1 } \autopar{1+ \cO \autopar{\frac{\delta^2}{h}}} \\
&=: h \, \Lambda_{m}^{int} (h,\delta, w_1) \autopar{1+ \cO \autopar{\frac{\delta^2}{h}}} \; ,
\end{align*}
with $w_1 \autopar{\tau} = e^{(m+1/2) \autopar{\delta_{1} - \tau}} u(\tau)$.\\
In the same way, we obtain for the exterior boundary,
\begin{align*}
    \cQ^{ext}_{m,h,\delta} \autopar{v} &= h^2 \int_{0}^{\delta_{2}} e^{\frac{2\tau}{h} \rho_{2} \partial_{n} \phi \autopar{\rho_{1}} + 2m \autopar{\delta_{2}-\tau} + \delta_{2}} \valabs{\partial_\tau w_2 \autopar{\tau}}^2 \frac{\diff \tau}{\rho_2}\autopar{1+ \cO \autopar{\frac{\delta^2}{h}}} \\
    &=: h \, \Lambda_{m}^{ext} \autopar{h,\delta, w_2}\autopar{1+ \cO \autopar{\frac{\delta^2}{h}}} \; ,
\end{align*}
with $\delta_2 = - \mathrm{ln} \autopar{1 - \frac{\delta}{\rho_2}}$ and $w_2 \autopar{\tau} = e^{(2m+1) \autopar{\delta_{2} - \tau}} v \autopar{\rho_2 e^{-\tau}}$.\\
Finally, we have the lower bound
\begin{equation*}
\left\| e^{-\phi / h} h \autopar{\partial_r + \frac{m + 1/2}{r} } v \right\|^2  \geq h \autopar{ \Lambda_{m}^{int} (h,\delta, w_1) + \Lambda_{m}^{ext} \autopar{h,\delta, w_2}} \autopar{1+ \cO \autopar{\frac{\delta^2}{h}}} \; .
\end{equation*}
Moreover, for any sufficiently regular function $w$ the maps 
\begin{equation*}
    \begin{array}{ccccc}
    \RR & \to & \RR^{*}_+ \\
    m & \mapsto & \Lambda_{m}^{ext} (h,\delta, w) 
    \end{array} \; \; \; \text{and} \; \; \; \begin{array}{ccccc}
    \RR & \to & \RR^{*}_+ \\
    m & \mapsto & \Lambda_{m}^{int} (h,\delta, w) 
    \end{array}
\end{equation*}
are increasing and decreasing respectively. Thus for all $m \in \RR_+$
\begin{equation*}
    \Lambda_{m}^{ext} (h,\delta,w) \geq \Lambda_{0}^{ext} (h,\delta,w) \; \; \; \text{and} \; \; \; \Lambda_{-m}^{int} (h,\delta, w) \geq \Lambda_{0}^{int} (h,\delta,w) \; .
\end{equation*}
By using \cite[Lemma A.1.]{barbaroux:hal-01889492} to the integrals $\Lambda_{m}^{int} (h,\delta,w_1)$, $\Lambda_{0}^{int} (h,\delta,w_1)$, $\Lambda_{m}^{ext} \autopar{h,\delta, w_2}$, $\Lambda_{0}^{ext} \autopar{h,\delta, w_2}$ we get for $m \leq 0$
\begin{align*}
    \Lambda_{0}^{int} \autopar{h,\delta, w_1 } &= \frac{h e^{-\delta_{1}}}{\rho_1} \int_{0}^{\delta_{1}} e^{\frac{2\tau}{h} \rho_{1} \partial_{n} \phi \autopar{\rho_{1}} } \valabs{\partial_\tau w_1 \autopar{\tau}}^2 \diff \tau\\
    &\geq e^{-\delta_{1}}\frac{2 \partial_n \phi \autopar{\rho_1} }{1 - e^{-2\delta_{1} \rho_{1} \partial_{n} \phi \autopar{\rho_{1}} /h }} \valabs{v\autopar{\rho_{1} + \delta }}^2 \; ,
\end{align*}
\begin{align*}
    \Lambda_{m}^{ext} (h,\delta, w_2) &= \frac{h e^{(2m+1)\delta_{2}}}{\rho_2} \int_{0}^{\delta_{2}} e^{\autopar{\frac{2\rho_{2} }{h}\partial_{n} \phi \autopar{\rho_{2}}  -2m } \tau } \valabs{\partial_\tau w_2 \autopar{\tau}}^2 \diff \tau \\
    &\geq e^{ \delta_{2}} \frac{2 \partial_n \phi \autopar{\rho_2} - 2mh/ \rho_2 }{e^{-2m\delta_{2}}- e^{-2\delta_{2} \rho_{2} \partial_{n} \phi \autopar{\rho_{2}} /h }} \valabs{v \autopar{\rho_{2} - \delta }}^2 \; ,
\end{align*}

and for $m \geq 0$

\begin{align*}
    \Lambda_{m}^{int} (h,\delta, w_1) &= \frac{h e^{-(2m+1)\delta_{1}}}{\rho_1} \int_{0}^{\delta_{1}} e^{\autopar{\frac{2\rho_{1} }{h}\partial_{n} \phi \autopar{\rho_{1}}  +2m } \tau } \valabs{\partial_\tau w_1 \autopar{\tau}}^2 \diff \tau \\
    &\geq e^{- \delta_{1}} \frac{2 \partial_n \phi \autopar{\rho_1} + 2mh/ \rho_1 }{e^{2m\delta_{1}}- e^{-2\delta_{1} \rho_{1} \partial_{n} \phi \autopar{\rho_{1}} /h }} \valabs{v \autopar{\rho_{1} + \delta }}^2 \; ,
\end{align*}
\begin{align*}
    \Lambda_{0}^{ext} \autopar{h,\delta, w_2 } &= \frac{h e^{\delta_{2}}}{\rho_2} \int_{0}^{\delta_{2}} e^{\frac{2\tau}{h} \rho_{2} \partial_{n} \phi \autopar{\rho_{2}} } \valabs{\partial_\tau w_2 \autopar{\tau}}^2 \diff \tau\\
    &\geq e^{\delta_{2}}\frac{2 \partial_n \phi \autopar{\rho_2} }{1 - e^{-2\delta_{2} \rho_{2} \partial_{n} \phi \autopar{\rho_{2}} /h }} \valabs{v\autopar{\rho_{2} - \delta }}^2 \; .
\end{align*}
This leads us to the result.
\end{proof}
Using Propositions \ref{propo:localisation_fct_propre} and \ref{propo:approx_holo}, we can replace $v_{h,m}$ by $\widetilde{\Pi}_{h,m}  v_{h,m}$ in the denominator of the Rayleigh quotient \eqref{eq.minmax.moment.reels}.
\begin{lemma} \label{lemme.premiere_maj_norme}
Let $\beta \in ]0,1/2[ $, $h>0$, $m\in\mathbb{R}$ and $v_{h,m}$ an eigenfunction associated to $\lambda_{1,m} (h)$.
\begin{equation*}
\left\| e^{-\phi/h} v_{h,m} \right\|_{L^{2} (\rho_{1}, \rho_{2})} \leq R_m (h) \, \left\| e^{-\phi/h} \widetilde{\Pi}_{h,m}  v_{h,m} \right\|_{L^{2} \autopar{\cI_{\min,h}}} \, ,
\end{equation*}
with $\cI_{\min, h} = \left[ r_{\min} - h^\beta , r_{\min} + h^\beta \right]$,
$$ 
R_m (h) = \autopar{\autopar{1 + \autopar{ f_{1,h} (m) + \autopar{f_{1,h} (m)}^2 }\cO \autopar{ h^\infty}} \; \sqrt{1 + f_{1,m} (h) \cO\autopar{h^\infty}}}^{-1} ,
$$
and $\widetilde{\Pi}_{h,m}$ defined in Definition \ref{def_projection_holo}. Moreover, $\cO\autopar{h^\infty}$ is independent of $m$.
\end{lemma}
\begin{proof}
Let $h>0$ and $m\in \RR$. Let $v_{h,m}$ an eigenfunction associated to $\lambda_{1,m} (h)$. \\
Proposition \ref{propo:localisation_fct_propre} ensures that for $\beta \in ]0 , 1/2[$, we have the existence of $h_0 >0$ such that for all $h \in ]0, h_0]$
$$
\autopar{1 + f_{1,h} (m) \cO\autopar{h^\infty}} \int_{\rho_{1}}^{\rho_{2}} e^{-2 \phi/h}  \vert v_{h,m}  \vert^2 = \int_{\cI_{\min , h}} e^{-2 \phi/h}  \vert v_{h,m} \vert^2 ,
$$
with $\cI_{\min , h} = [r_{\min} - h^\beta , r_{\min} + h^\beta]$. \\
Using Proposition \ref{propo:localisation_fct_propre}, we have the following inequalities
\begin{align*}
\left\| e^{-\phi/h} v_{h,m} \right\|_{L^{2} \autopar{\cI_{\min,h}}}  &\leq \left\| e^{-\phi/h} \widetilde{\Pi}_{h,m}  v_{h,m} \right\|_{L^{2} \autopar{\cI_{\min,h}}}  +   \left\| e^{-\phi/h} \autopar{id - \widetilde{\Pi}_{h,m} }v_{h,m} \right\|_{L^{2} \autopar{\cI_{\min,h}}} \\
&\leq \left\| e^{-\phi/h} \widetilde{\Pi}_{h,m}  v_{h,m} \right\|_{L^{2} \autopar{\cI_{\min,h}}}+ \left\| e^{-\phi/h} \autopar{id - \widetilde{\Pi}_{h,m} }v_{h,m} \right\|_{L^{2} \autopar{\cI_{\delta}}} 
\end{align*}
with $\delta>0$ sufficiently small, $C>0$ and $\cI_\delta = [\rho_1 + \delta , \rho_2 - \delta]$.\\
However, according to Propostion \ref{propo:approx_holo}
\begin{equation*}
\left\| e^{-\phi/h} \autopar{id - \widetilde{\Pi}_{h,m} }v_{h,m} \right\|_{L^{2} \autopar{\cI_{\delta}}}  \leq \autopar{ f_{1,h} (m) + \autopar{f_{1,h} (m)}^2 }\cO \autopar{ h^\infty} \left\| e^{-\phi/h} v_{h,m} \right\|_{L^{2} \autopar{\cI_{\min,h}}} ,
\end{equation*}
with $\cO \autopar{h^\infty}$ independent of $m$.\\
Then
\begin{equation*}
\left\| e^{-\phi/h} v_{h,m} \right\|_{L^{2} (\rho_{1}, \rho_{2})} \leq R_m (h) \left\| e^{-\phi/h} \widetilde{\Pi}_{h,m}  v_{h,m} \right\|_{L^{2} \autopar{\cI_{\min,h}}} ,
\end{equation*}
with $R_m (h) = \autopar{\autopar{1 + \autopar{ f_{1,h} (m) + \autopar{f_{1,h} (m)}^2 }\cO \autopar{ h^\infty}} \; \sqrt{1 + f_{1,m} (h) \cO\autopar{h^\infty}}}^{-1} $.
\end{proof}
Proposition \ref{propo.Dirac} ensures that $ \widetilde{\Pi}_{h,m}  v_{h,m} = \alpha_m r^{m+1/2}$ with $\alpha_m \in \CC^*$. The denominator can be again simplified as follows
\begin{lemma} \label{lemme.seconde_maj_norme}
Let $\beta \in ]1/3, 1/2[$, $h>0$, $m\in\mathbb{R}$ and $v_{h,m}$ an eigenfunction associated to $\lambda_{1,m} (h)$.
\begin{multline*}
 \left\| e^{-\phi/h} \widetilde{\Pi}_{h,m}  v_{h,m} \right\|_{L^{2} \autopar{\cI_{\min, h}}}^2 \leq \\
\valabs{\alpha_m }^2 \sqrt{\frac{\pi}{\phi''_{\min}}} \autopar{r_{\min} + \mathrm{sgn}(2m+1) h^{\beta} }^{2m+1} \sqrt{h} \; e^{-2 \phi_{\min} / h} \autopar{1 + o(1)} ,
\end{multline*}
with $\cI_{\min, h} = \left[ r_{\min} - h^\beta , r_{\min} + h^\beta \right]$, $\alpha_m \in \CC^*$ such that  $ \widetilde{\Pi}_{h,m}  v_{h,m} = \alpha_m r^{m+1/2}$.\\
Moreover, $o(1)$ is independent of $m$.
\end{lemma}
\begin{proof}
Let $\beta \in ]1/3, 1/2[$, $h >0$ and $v_{h,m}$ an eigenfunction associated to $\lambda_{1,m} (h)$. \\
According to Proposition \ref{propo.Dirac}, for all $m \in \RR$ there exists $\alpha_m \in \CC^*$ such that $ \widetilde{\Pi}_{h,m}  v_{h,m} = \alpha_m r^{m+1/2}$. Then, by Lemma \ref{lemme.premiere_maj_norme}, we have
\begin{equation*}
    \left\| e^{-\phi/h} \widetilde{\Pi}_{h,m}  v_{h,m} \right\|_{L^{2} \autopar{\cI_{\min, h}}}^2 = \valabs{\alpha_m}^2 \int_{\cI_{\min , h}}   e^{-2\phi / h}  r^{2m +1} \diff r ,
\end{equation*}
with $\cI_{\min , h} = \left[ r_{\min} - h^\beta , r_{\min} + h^\beta \right]$. \\
As in the proof of Lemma \ref{lem:minoration_denominateur}, we obtain
$$
\phi ( r) = \phi_{\min} + \frac{\autopar{r-r_{\min}}^2}{2} \phi''_{\min} + \cO \autopar{h^{3\beta}}.
$$
Thus,
\begin{align*}
    \int_{\cI_{\min , h}}  & e^{-2\phi / h} r^{2m +1} \diff r \leq   e^{-2\phi_{\min} / h} \int_{\cI_{\min , h}}   e^{- \frac{\autopar{r - r_{\min}}^{2}}{h} \phi_{\min}^{''} }  r^{2m +1} \diff r \autopar{1 + \cO \autopar{h^{3\beta -1}}} \\
    & \leq \sqrt{\frac{h}{\phi''_{\min}}} \autopar{r_{\min} + \mathrm{sgn}(2m+1) h^{\beta} }^{2m+1} e^{-2\phi_{\min} / h} \int_\RR e^{-\tau^{2}} \diff r  \autopar{1 + \cO \autopar{h^{3\beta -1}}} .
\end{align*}
We deduce that
\begin{multline*}
\int_{\cI_{\min , h}}   e^{-2\phi / h} r^{2m +1} \diff r \leq \\
\sqrt{\frac{\pi}{\phi''_{\min}}} \autopar{r_{\min} + \mathrm{sgn}(2m+1) h^{\beta} }^{2m+1} \sqrt{h} \; e^{-2 \phi_{\min} / h} \autopar{1 + \cO \autopar{h^{3\beta -1}}} ,
\end{multline*}
and the conclusion follows.
\end{proof}

\subsection{Uniform convergence on any compact}
In this part we prove Proposition \ref{propo.CVU.coercivite} (\ref{pt1.propCVU_coercivite}).
\begin{proposition} \label{proposition.cvu}
The sequence $\autopar{f_{1,h} }_{h \in ]0, 1] }$ is uniformly convergent on any compact to $f\colon \RR \rightarrow \RR^+$ when $h$ tends to $0$ where
\begin{equation*} 
f\colon m \longmapsto 2 \sqrt{\frac{\phi''_{\min}}{\pi}} \autopar{\partial_n\phi\autopar{\rho_1} \autopar{\frac{\rho_1}{r_{\min}} }^{2m+1} + \partial_n\phi\autopar{\rho_2} \autopar{\frac{\rho_2}{r_{\min}} }^{2m+1}  } \, .
\end{equation*}
\end{proposition}
When we take a compact, the upper bound of Proposition \ref{propo.borne_sup_lambda1m} ensures the validity of the following lemma.
\begin{lemma} \label{lemme.majoration_sur_compact}
There exists $h_0 > 0$ such that for all $ h \in ]0, h_0]$ and for all $m \in \RR$, we have
$$
0 < f_{1,h} (m) \leq f(m) \; (1+o_{h \rightarrow 0}(1)),
$$
with $o_{h \rightarrow 0}(1)$ uniform in $m$ on all compacts of $\RR$.
\end{lemma}
Let us start by noting the following result.
\begin{lemma} \label{lemme:injectivite_proj_1}
For all $m$ included in a compact of $\RR$, the projection $\widetilde{\Pi}_{h,m}$, defined in Definition \ref{def_projection_holo}, is injective when it acts on the eigenspace associated with $\lambda_{1,m} (h)$.
\end{lemma}
\begin{proof}
Let $\cM$ a compact of $ \RR$. Let $h>0$, $m\in \cM$ and $v_{h,m}$ an eigenfunction associated to $\lambda_{1,m} (h)$ such that $\widetilde{\Pi}_{h,m} v_{h,m} = 0$.\\
According to Proposition \ref{propo:approx_holo}, there exist $C, h_0, \delta_0 >0$ such that for all $\delta \in ]0,\delta_0]$ and $h \in ]0 , h_0]$, we have
\begin{align*}
\left\|e^{-\phi/h} v_{h,m}\right\|_{L^{2} (\cI_\delta )} & \leq  \left\|e^{-\phi/h} \autopar{\Id - \widetilde{\Pi}_{h,m}} v_{h,m}\right\|_{L^{2} (\cI_\delta )} + \left\|e^{-\phi/h} \widetilde{\Pi}_{h,m} v_{h,m}\right\|_{L^{2} (\cI_\delta )} \\
& \leq C e^{\phi_{\min}/h} \left( f_{1,h} (m) + \left( f_{1,h} (m) \right)^2 \right) \left\|e^{-\phi/h} v_{h,m}\right\|_{L^{2} (\cI_\delta )} ,
\end{align*}
with $\widetilde{\Pi}_{h,m}$ defined in Definition \ref{def_projection_holo} and $\cI_\delta = [\rho_1 + \delta , \rho_2 - \delta]$. \\
Using Lemma \ref{lemme.majoration_sur_compact}, there exists another constant $C\left( \cM \right)>0$ such that 
$$
\left\|e^{-\phi/h} v_{h,m}\right\|_{L^{2} (\cI_\delta )} \autopar{ 1 - C e^{\phi_{\min}/h}} \leq 0 .
$$
Thus $v_{h,m} = 0 $ on $\cI_{\delta}$ which ensures the injectivity of $\widetilde{\Pi}_{h,m}$.
\end{proof}
The estimate of the $L^2$-norm follows directly from Lemmas \ref{lemme.premiere_maj_norme} \& \ref{lemme.seconde_maj_norme} and from Lemma \ref{lemme.majoration_sur_compact}. 
\begin{lemma} \label{lemme.cvu_majoration_L2norme}
Let $h>0$, $m\in\mathbb{R}$ and $v_{h,m}$ an eigenfunction associated to $\lambda_{1,m} (h)$.
\begin{equation*}
\left\| e^{-\phi/h} v_{h,m} \right\|^2 \leq \valabs{\alpha_m}^2 \sqrt{\frac{\pi}{\phi''_{\min}}} \, r_{\min}^{2m+1} \sqrt{h} \; e^{-2 \phi_{\min} / h} \autopar{1 + o \autopar{1}}   \, ,
\end{equation*}
with $\alpha_m \in \CC^*$ such that $ \widetilde{\Pi}_{h,m}  v_{h,m} = \alpha_m r^{m+1/2}$.\\
Moreover, $o_{h \rightarrow 0}(1)$ is uniform in $m$ on all compacts of $\RR$.
\end{lemma}
\begin{proof}
It is enough to apply Lemmas \ref{lemme.premiere_maj_norme} and \ref{lemme.seconde_maj_norme} by noticing that when we take a compact we have
$$
R_m (h) = 1+ o_{h \rightarrow 0} (1),
$$
with $o_{h \rightarrow 0}(1)$ uniform in $m$ on all compacts of $\RR$.
\end{proof}
Under the compactness assumptions, we can simplify the lower bound of Lemma \ref{lemme_reduction_num}.
\begin{lemma} \label{lemme1_cvu_min_energie}
There exist $h_0 , \delta_0 > 0$ such that for all $ h \in ]0, h_0]$, for all $ \delta \in ]0, \delta_0]$, for all $m \in \RR$ and for all $v \in H^1_0$ we have
\begin{equation*}
\left\| e^{-\phi / h} h \autopar{ \partial_r + \frac{m + 1/2}{r} } v \right\|^2  \geq 2 h \; \cN_{\partial_{n} \phi} (V)^2  \autopar{1+ o \autopar{1}} ,
\end{equation*}
when $\delta / h \rightarrow + \infty$ and $\delta^2 / h \rightarrow 0$ with $\cN_{\partial_{n} \phi} ( \cdot )$ defined in Notation \ref{norme_l2_poids} and 
$$
V= \autopar{ v(\rho_1 + \delta) , v( \rho_2 - \delta ) } .
$$
Moreover, $o_{h \rightarrow 0}(1)$ is uniform in $m$ on all compacts of $\RR$.
\end{lemma}
We can then minimize the energy, uniformly in $m$, by the trace of the projection of the eigenfunctions on the kernel of $d_{h,m}^\times$.
\begin{lemma} \label{lemme2_cvu_min_energie}
There exists $h_0 > 0$ such that for all $ h \in ]0, h_0]$, for all $m \in \RR$ and for $v_{h,m}$ an eigenfunction associated to $\lambda_{1,m} (h)$ we have
\begin{equation*}
\lambda_{1,m} (h) \, \left\| e^{-\phi/h} v_{h,m} \right\|^2 \geq 2h \valabs{\alpha_m}^2 \autopar{\partial_n \phi \autopar{\rho_1} \rho_1^{2m+1} + \partial_n \phi \autopar{\rho_2} \rho_2^{2m+1} } (1 + o(1)),
\end{equation*}
with $\alpha_m \in \CC^*$ such that $ \widetilde{\Pi}_{h,m}  v_{h,m} = \alpha_m r^{m+1/2}$. \\
Moreover, $o_{h \rightarrow 0}(1)$ is uniform in $m$ on all compacts of $\RR$.
\end{lemma}
\begin{proof} Let $\cM$ a compact of $\RR$. Let $h,\delta>0$, $m \in \cM$ and $v_{h,m}$ an eigenfunction associated to $\lambda_{1,m} (h)$. Let us further assume that $\delta / h \rightarrow + \infty$ $\delta^2 / h \rightarrow 0$. 
\begin{enumerate}[label = \roman*)]
\item When we take $v_{h,m}$ as an eigenfunction, Lemma \ref{lemme1_cvu_min_energie} becomes 
\begin{equation*}
\sqrt{\lambda_{1,m} (h)} \, \left\|e^{-\phi/h} v_{h,m}\right\| \geq \sqrt{2h} \, \cN_{\partial_{n} \phi} \autopar{V_{h,m}} (1 + o(1)) , 
\end{equation*}
with $\cN_{\partial_{n} \phi} ( \cdot )$ the norm defined in Notation \ref{norme_l2_poids} and
$$
V_{h,m}= \autopar{ v_{h,m}(\rho_1 + \delta) , v_{h,m}( \rho_2 - \delta ) }.
$$ 
\item \label{cvu_point_ii} According to Proposition \ref{propo:approx_holo} and Lemma \ref{lemme.majoration_sur_compact}, we have
\begin{equation*}
\cN_{\partial_{n} \phi} \left( U_{h,m} \right)  \leq C h^{-3/4} \sqrt{\lambda_{1,m} (h)}  \left\|e^{-\phi/h} v_{h,m}\right\|_{L^{2} \autopar{\cI_{\delta}}} (1 + o(1)) \; ,
\end{equation*}    
with $\cI_\delta = [\rho_1 + \delta , \rho_2 - \delta ]$, $o(1)$ uniform in $m$ on all compacts of $\RR$ and
\begin{equation*}
U_{h,m} = \left( e^{-2\phi(\rho_{1}+ \delta)/h}\left( \Id - \widetilde{\Pi}_{h,\delta}  \right)v_{h,m} \autopar{\rho_1 + \delta}, e^{-2\phi(\rho_{2}- \delta)/h}\left( \Id - \widetilde{\Pi}_{h,\delta}  \right)v_{h,m} \autopar{\rho_2 - \delta} \right) .
\end{equation*}
Using the Taylor expansion of $\phi$ in the neighborhood of $\rho_1$ and $\rho_2$ and by writing $N = \min \autopar{\partial_n \phi (\rho_1) , \partial_n \phi (\rho_2)} >0 $, we get
\begin{equation*}
\cN_{\partial_{n} \phi}  \autopar{\autopar{\Id - \Pi}V_{h,m}}  \leq C h^{-3/4} e^{- N \delta /h} \sqrt{\lambda_{1,m} (h)}  \left\|e^{-\phi/h} v_{h,m}\right\|_{L^{2} \autopar{\cI_{\delta}}} (1 + o(1)) \; ,
\end{equation*}
with $\Pi V_{h,m}= \autopar{ \widetilde{\Pi}_{h,m }v_{h,m}(\rho_1 + \delta) , \widetilde{\Pi}_{h,m }v_{h,m}( \rho_2 - \delta ) } $. 
\item Then the triangular inequality ensures that
\begin{equation*}
\cN_{\partial_{n} \phi} \autopar{\autopar{\Id - \Pi}V_{h,m}} \geq \valabs{\cN_{\partial_{n} \phi} \autopar{V_{h,m}} - \cN_{\partial_{n} \phi} \autopar{\Pi V_{h,m}} } .
\end{equation*}
Thus, by \ref{cvu_point_ii}
\begin{equation*}
\cN_{\partial_{n} \phi} \autopar{V_{h,m}} \geq \cN_{\partial_{n} \phi} \autopar{\Pi V_{h,m}} - \cO \autopar{h^\infty} \sqrt{\lambda_{1,m} (h)} \left\|e^{-\phi/h} v_{h,m}\right\|_{L^{2} \autopar{\cI_{\delta}}} .
\end{equation*}
Proposition \ref{propo:localisation_fct_propre} allows us to obtain for $m$ bounded that
\begin{equation*}
\left\|e^{-\phi / h } v_{h,m}\right\|_{L^{2} \autopar{\cI_{\delta}}}  = \left\|e^{-\phi / h } v_{h,m}\right\| (1+o(1)) \, ,
\end{equation*}
with $o_{h \rightarrow 0}(1)$ uniform in $m$ on all compacts of $\RR$.\\
Then,
\begin{equation*}
\lambda_{1,m} (h) \, \left\| e^{-\phi/h} v_{h,m} \right\|^2 \geq 2h \valabs{\alpha_m}^2 \autopar{\partial_n \phi \autopar{\rho_1} \rho_1^{2m+1} + \partial_n \phi \autopar{\rho_2} \rho_2^{2m+1} } (1 + o(1)),
\end{equation*}
with $\alpha_m \in \CC^*$ such that $ \widetilde{\Pi}_{h,m}  v_{h,m} = \alpha_m r^{m+1/2}$ and $o_{h \rightarrow 0}(1)$ uniform in $m$ on all compacts of $\RR$.
\end{enumerate}
\end{proof}
\begin{proof}[Proof of Proposition \ref{proposition.cvu}] Recall that
$$
\lambda_{1,m} (h) = f_{1,h} (m) \sqrt{h} \, e^{2 \phi_{\min}/ h} .
$$
By combining Lemmas \ref{lemme.cvu_majoration_L2norme} and \ref{lemme2_cvu_min_energie}, we have the existence of $h_0 >0$ such that for all $h \in ]0, h_0]$ 
$$
 f_{1,h} (m) \geq f(m) \; (1+o_{h \rightarrow 0}(1)),
$$
with $o_{h \rightarrow 0}(1)$ uniform in $m$ on all compacts of $\RR$. \\
Lastly, Lemma \ref{lemme.majoration_sur_compact} gives us
$$
 f_{1,h} (m) = f(m) \; (1+o_{h \rightarrow 0}(1)),
$$
with always $o_{h \rightarrow 0}(1)$ uniform in $m$ on all compacts of $\RR$. The result follows.
\end{proof}
\subsection{Weak coercivity}
In this part we prove Proposition \ref{propo.CVU.coercivite} (\ref{pt2.propCVU_coercivite}).
\begin{proposition} \label{propo_coercivite_faible}
Let $K>0$. There exist $h_0>0$ and $g : \RR \rightarrow \RR^+ $ coercive such that for all $h\in ]0, h_0] $ and $m \in \RR$ satisfying
\[
f_{1,h}(m)\leq K\,,
\]
we have
$f_{1,h} (m) \geq g (m)$.
\end{proposition}
Similarly to Lemma \ref{lemme:injectivite_proj_1}, we have
\begin{lemma} \label{lemme:injectivite_proj_2}
Let $K>0$ and $h_0>0$. For all $h\in ]0, h_0] $ and $m \in \RR$ satisfying
\[
f_{1,h}(m)\leq K\,,
\]
the projection $\widetilde{\Pi}_{h,m}$, definied in Definition \ref{def_projection_holo}, is injective when it acts on the eigenspace associated to $\lambda_{1,m} (h)$.
\end{lemma}
\begin{proof}
Let $K>0$ and $h_0>0$. For all $h\in ]0, h_0] $ and $m \in \RR$ satisfying
\[
f_{1,h}(m)\leq K\,,
\]
consider $v_{h,m}$ an eigenfunction associated to $\lambda_{1,m} (h)$ satisfying $\widetilde{\Pi}_{h,m} v_{h,m} = 0$.\\
According to Proposition \ref{propo:approx_holo}, there exist $C, h_0, \delta_0 >0$ such that for all $\delta \in ]0,\delta_0]$ and $h \in ]0 , h_0]$, we have
\begin{align*}
\left\|e^{-\phi/h} v_{h,m}\right\|_{L^{2} (\cI_\delta )} & \leq  \left\|e^{-\phi/h} \autopar{\Id - \widetilde{\Pi}_{h,m}} v_{h,m}\right\|_{L^{2} (\cI_\delta )} + \left\|e^{-\phi/h} \widetilde{\Pi}_{h,m} v_{h,m}\right\|_{L^{2} (\cI_\delta )} \\
& \leq C e^{\phi_{\min}/h} \left( f_{1,h} (m) + \left( f_{1,h} (m) \right)^2 \right) \left\|e^{-\phi/h} v_{h,m}\right\|_{L^{2} (\cI_\delta )} ,
\end{align*}
with $\widetilde{\Pi}_{h,m}$ given in Definition \ref{def_projection_holo} and $\cI_\delta = [\rho_1 + \delta , \rho_2 - \delta]$. \\
Then, there exists another constant $C\left( K \right)>0$ such that
$$
\left\|e^{-\phi/h} v_{h,m}\right\|_{L^{2} (\cI_\delta )} \autopar{ 1 - C e^{\phi_{\min}/h}} \leq 0 .
$$
Then $v_{h,m} = 0 $ on $\cI_{\delta}$ that is provides the injectivity of $\widetilde{\Pi}_{h,m}$.
\end{proof}

\begin{remark}
Taking $\beta = 2/5$ in Lemma \ref{lemme.seconde_maj_norme}.
\end{remark}
After that, the $L^2$-norm can be estimated as follows.
\begin{lemma} \label{lemme.coercivite_majoration_L2norme}
Let $K>0$. There exists $h_0>0$ such that for all $h\in ]0, h_0] $ and $m \in \RR$ satisfying
\[
f_{1,h}(m)\leq K\,,
\]
we have for all $v_{h,m}$ eigenfunction associated to $\lambda_{1,m} (h)$
\begin{equation*}
\left\| e^{-\phi/h} v_{h,m} \right\|^2 \leq \valabs{\alpha_m}^2 \sqrt{\frac{\pi}{\phi''_{\min}}} \, \autopar{r_{\min} + \mathrm{sgn}(2m+1) h^{2/5} }^{2m+1} \sqrt{h} \; e^{-2 \phi_{\min} / h} \autopar{1 + o \autopar{1}}   \, ,
\end{equation*}
with $\alpha_m \in \CC^*$ such that $ \widetilde{\Pi}_{h,m}  v_{h,m} = \alpha_m r^{m+1/2}$. \\
Moreover, $o_{h \rightarrow 0}(1)$ is uniform in $m$.
\end{lemma}
\begin{proof}
It is enough to apply Lemma \ref{lemme.premiere_maj_norme} \& \ref{lemme.seconde_maj_norme} by noting that under these assumptions
$$
R_m (h) = 1+ o_{h \rightarrow 0} (1),
$$
with $o_{h \rightarrow 0}(1)$ uniform in $m$.
\end{proof}
\begin{lemma} \label{lemme1_coercivite_minoration_energie}
There exist $h_0 , \delta_0 >0$ such that for all $h \in ]0,h_0]$, for all $\delta \in ]0,\delta_0 ]$ and for all $v \in H_0^1$
\begin{equation*}
\left\| e^{-\phi / h} h \autopar{\partial_r + \frac{m + 1/2}{r} } v \right\|^2  \geq h \;  \cN_{\partial_{n} \phi} \autopar{V^0} \autopar{1+ o(1)} \; ,
\end{equation*}
when $\delta / h \rightarrow + \infty$ and $\delta^2 / h \rightarrow 0$ with $\cN_{\partial_{n} \phi} ( \cdot )$ defined in Notation \ref{norme_l2_poids} and 
$$
V^0 = \autopar{ v(\rho_1 + \delta) \mathds{1}_{]-\infty,0[} (m) , v( \rho_2 - \delta ) \mathds{1}_{]0, +\infty[} (m)  }.
$$
Moreover, $o_{h \rightarrow 0}(1)$ is uniform in $m$.
\end{lemma}
\begin{proof}
Let $h , \delta >0$, $m \in \RR$ and $v\in H_0^1$. We can reduce the quantities $E_m^{int} \autopar{h,\delta}$ and $E_m^{ext} \autopar{h,\delta}$ defined in Lemma \ref{lemme_reduction_num} by using the mean value theorem. It is enough to note that for $m\in \RR_+^*$ there exists $\xi_{\mathrm{int}}\in (-2\delta_1 \rho_1 \partial_n \phi (\rho_1)/h,0)$ such that
\begin{equation*}
    E^{int}_m (h,\delta) = e^{- \delta_{1}} \frac{2 \partial_n \phi \autopar{\rho_1} + (2mh/ \rho_1) \mathds{1}_{]0, +\infty[} (m) }{e^{2m\delta_{1} \mathds{1}_{]0, +\infty[} (m)}- e^{-2\delta_{1} \rho_{1} \partial_{n} \phi \autopar{\rho_{1}} /h }} = \frac{h e^{-\delta_{1}}}{\rho_1 \delta_1} e^{-\xi_{\mathrm{int}}} > 0 \, ,
\end{equation*}
and for $m\in \RR_-^*$ there exists $\xi_{\mathrm{ext}}\in (-2\delta_2 \rho_2 \partial_n \phi (\rho_2)/h,0)$ such that
\begin{align*}
E^{ext}_m (h,\delta) &= e^{ \delta_{2}} \frac{2 \partial_n \phi \autopar{\rho_2} - (2mh/ \rho_2) \mathds{1}_{]-\infty,0[} (m) }{e^{-2m\delta_{2} \mathds{1}_{]-\infty,0[} (m)}- e^{-2\delta_{2} \rho_{2} \partial_{n} \phi \autopar{\rho_{2}} /h }} = \frac{h e^{\delta_{2}}}{\rho_2 \delta_2} e^{-\xi_{\mathrm{ext}}} > 0 \, ,
\end{align*}
with $\delta_1 = \mathrm{ln} \autopar{1+\frac{\delta}{\rho_1}}$, $\delta_2 = - \mathrm{ln} \autopar{1 - \frac{\delta}{\rho_2}}$. \\
If we assume that $\delta /h \rightarrow +\infty$,
\begin{align*}
    E^{int}_m (h,\delta) &\geq e^{- \delta_{1}}\frac{2 \partial_n \phi \autopar{\rho_1} }{1 - e^{-2\delta_{1} \partial_{n} \phi \autopar{\rho_{1}} /h }} \mathds{1}_{]-\infty,0[} (m) \\
    & \geq   \partial_n \phi \autopar{\rho_1} \mathds{1}_{]-\infty,0[} (m) (1+o(1)) \; ,
\end{align*}
and
\begin{align*}
    E^{ext}_m (h,\delta)  &\geq e^{\delta_{2}}\frac{2 \partial_n \phi \autopar{\rho_2} }{1- e^{-2\delta_{2} \partial_{n} \phi \autopar{\rho_{2}} / h}} \mathds{1}_{]0, +\infty[} (m)  \\
    & \geq  \partial_n \phi \autopar{\rho_2} \mathds{1}_{]0, +\infty[} (m)  (1+ o(1)) \; .
\end{align*}
Therefore, by assuming $\delta^2 / h \rightarrow 0$, Lemma \ref{lemme_reduction_num} ensures the existence of $h_0 , \delta_0 >0$ such that for all $h \in ]0,h_0]$, for all $\delta \in ]0,\delta_0 ]$ and for all $v \in H_0^1$
\begin{equation*}
\left\| e^{-\phi / h} h \autopar{\partial_r + \frac{m + 1/2}{r} } v \right\|^2  \geq h \; \cN_{\partial_{n} \phi} \autopar{V^0} \autopar{1+ o(1)} \; ,
\end{equation*}
with $\cN_{\partial_{n} \phi} (\cdot)$ defined in Notation \ref{norme_l2_poids} and 
$$
V^0 = \autopar{ v(\rho_1 + \delta) \mathds{1}_{]-\infty,0[} (m) , v( \rho_2 - \delta ) \mathds{1}_{]0, +\infty[} (m)  }.
$$
\end{proof}
\begin{lemma} \label{lemme2_coercivite_minoration_energie}
Let $K>0$. There exists $h_0>0$ such that for all $h\in ]0, h_0] $ and $m \in \RR$ satisfying
\[
f_{1,h}(m)\leq K\,,
\]
we have for all $v_{h,m}$ eigenfunction associated to $\lambda_{1,m} (h)$
\begin{multline*}
\lambda_{1,m} (h) \, \left\| e^{-\phi/h} v_{h,m} \right\|^2 \geq \\
h N \valabs{\alpha_m}^2 \autopar{ \autopar{\rho_1 + \delta }^{2m+1} \mathds{1}_{]- \infty , 0[} (m)  +  \autopar{\rho_2 - \delta }^{2m+1} \mathds{1}_{]0, +\infty[} (m)  } (1 + o(1)),
\end{multline*}
when $\delta^2 /h \rightarrow 0 $ and $\delta/h \rightarrow +\infty$ with $N = \mathrm{min} \autopar{\partial_n \phi \autopar{\rho_1 } , \partial_n \phi \autopar{\rho_2 }  } > 0 $ and $\alpha_m \in \CC^*$ such that $ \widetilde{\Pi}_{h,m}  v_{h,m} = \alpha_m r^{m+1/2}$.\\
Moreover, $o_{h \rightarrow 0}(1)$ is uniform in $m$ and depends on $K$.
\end{lemma}
\begin{proof} Let $K>0$ and $h_0>0$. Consider $h \in ]0, h_0]$ and $m\in\mathbb{R}$ satisfying 
\[
f_{1,h}(m)\leq K\,.
\]
Let $\delta>0$ and $v_{h,m}$ an eigenfunction associated to $\lambda_{1,m} (h)$. Let us also assume that $\delta / h \rightarrow + \infty$ and $\delta^2 / h \rightarrow 0$. 
\begin{enumerate}[label = \roman*)]
\item When we take $v_{h,m}$ an eigenfunction, Lemma \ref{lemme1_coercivite_minoration_energie} becomes 
\begin{equation*}
\sqrt{\lambda_{1,m} (h)} \, \left\|e^{-\phi/h} v_{h,m}\right\| \geq \sqrt{h} \, \cN_{\partial_{n} \phi} \autopar{V^0_{h,m}} (1 + o(1)) , 
\end{equation*}
with $\cN_{\partial_{n} \phi} ( \cdot )$ the norm defined in Notation \ref{norme_l2_poids} and
$$
V^0_{h,m} = \autopar{ v_{h,m}(\rho_1 + \delta) \mathds{1}_{]-\infty,0[} (m) , v_{h,m}( \rho_2 - \delta ) \mathds{1}_{]0, +\infty[} (m)  }.
$$
\item \label{cvu_point_ii_bis} Proposition \ref{propo:approx_holo} ensures the existence of a constant $C_K >0$ such that 
\begin{equation*}
\cN_{\partial_{n} \phi} \left( U_{h,m} \right)  \leq C_K h^{-3/4} \sqrt{\lambda_{1,m} (h)}  \left\|e^{-\phi/h} v_{h,m}\right\|_{L^{2} \autopar{\cI_{\delta}}} (1 + o(1)) \; ,
\end{equation*}    
with $\cI_\delta = [\rho_1 + \delta , \rho_2 - \delta ]$, $o(1)$ uniform in $m$ and dependent on $K$ and
\begin{equation*}
U_{h,m} = \left( e^{-2\phi(\rho_{1}+ \delta)/h}\left( \Id - \widetilde{\Pi}_{h,\delta}  \right)v_{h,m} \autopar{\rho_1 + \delta}, e^{-2\phi(\rho_{2}- \delta)/h}\left( \Id - \widetilde{\Pi}_{h,\delta}  \right)v_{h,m} \autopar{\rho_2 - \delta} \right) .
\end{equation*}
Using the Taylor expansion of $\phi$ in the neighborhood of $\rho_1$ and $\rho_2$ and by writing $N = \min \autopar{\partial_n \phi (\rho_1) , \partial_n \phi (\rho_2)} >0 $, we get
\begin{equation*}
\cN_{\partial_{n} \phi}  \autopar{\autopar{\Id - \Pi}V_{h,m}}  \leq C_K h^{-3/4} e^{- N \delta /h} \sqrt{\lambda_{1,m} (h)}  \left\|e^{-\phi/h} v_{h,m}\right\|_{L^{2} \autopar{\cI_{\delta}}} (1 + o(1)) \; ,
\end{equation*}
with $\autopar{ \Id -\Pi}  V_{h,m}= \autopar{ \left( \Id - \widetilde{\Pi}_{h,\delta}  \right)v_{h,m}(\rho_1 + \delta) , \left( \Id - \widetilde{\Pi}_{h,\delta}  \right)v_{h,m}( \rho_2 - \delta ) } $. 
\item Then the triangular inequality gives
\begin{equation*}
\cN_{\partial_{n} \phi} \autopar{\autopar{\Id - \Pi}V_{h,m}} \geq \cN_{\partial_{n} \phi} \autopar{\autopar{\Id - \Pi}V_{h,m}^0} \geq \valabs{\cN_{\partial_{n} \phi} \autopar{V_{h,m}^0} - \cN_{\partial_{n} \phi} \autopar{\Pi V_{h,m}^0} } ,
\end{equation*}
with $V^0_{h,m} = \autopar{ \widetilde{\Pi}_{h,\delta} v_{h,m}(\rho_1 + \delta) \mathds{1}_{]-\infty, 0[} (m) , \widetilde{\Pi}_{h,\delta} v_{h,m}( \rho_2 - \delta ) \mathds{1}_{]0, + \infty[} (m) } $. \\
Thus, by \ref{cvu_point_ii}
\begin{equation*}
\cN_{\partial_{n} \phi} \autopar{V_{h,m}^0} \geq \cN_{\partial_{n} \phi} \autopar{\Pi V_{h,m}^0} - \cO \autopar{h^\infty} \sqrt{\lambda_{1,m} (h)} \left\|e^{-\phi/h} v_{h,m}\right\|_{L^{2} \autopar{\cI_{\delta}}} .
\end{equation*}
The uniform bound assumption on $f_{1,h} (\cdot) $ and Proposition \ref{propo:localisation_fct_propre} also ensure the localization of eigenfunction, in particular
\begin{equation*}
\left\|e^{-\phi / h } v_{h,m}\right\|_{L^{2} \autopar{\cI_{\delta}}} = \left\|e^{-\phi / h } v_{h,m}\right\| (1+o(1))  \, ,
\end{equation*}
with $o(1)$ uniform in $m$ and dependent on $K$.\\
Then,
\begin{multline*}
\lambda_{1,m} (h) \, \left\| e^{-\phi/h} v_{h,m} \right\|^2 \geq \\
h N \valabs{\alpha_m}^2 \autopar{ \autopar{\rho_1 + \delta }^{2m+1} \mathds{1}_{]- \infty , 0[} (m)  +  \autopar{\rho_2 - \delta }^{2m+1} \mathds{1}_{]0, +\infty[} (m)  } (1 + o(1)),
\end{multline*}
with $\alpha_m \in \CC^*$ such that $ \widetilde{\Pi}_{h,m}  v_{h,m} = \alpha_m r^{m+1/2}$ and $o_{h \rightarrow 0}(1)$ uniform in $m$ and dependent on $K$.
\end{enumerate}
\end{proof}
\begin{proof}[Proof of Proposition \ref{propo_coercivite_faible}] ~~ \\
Let $K>0$ and $h_0>0$. Let us consider $h \in ]0, h_0]$ and $m\in\mathbb{R}$ satisfying 
\[
f_{1,h}(m)\leq K\,.
\]
Let us take $v_{h,m}$ an eigenfunction associated to $\lambda_{1,m} (h) $. \\
By combining Lemmas \ref{lemme.coercivite_majoration_L2norme} and \ref{lemme2_coercivite_minoration_energie} we have the existence of $\delta_0 >0$ such that for all $\delta \in ]0, \delta_0]$ satisfying $\delta /h \rightarrow + \infty$, $\delta^2 / h \rightarrow 0 $,
\begin{multline*}
f_{1,h} (m) \geq N \left( \autopar{\frac{\rho_1 + \delta}{r_{\min} + sgn (2m+1) h^{2/5}}}^{2m+1} \mathds{1}_{m\leq 0} \right.\\
\left. + \autopar{\frac{\rho_2 - \delta}{r_{\min} + sgn (2m+1) h^{2/5}}}^{2m +1}  \mathds{1}_{m\geq 0} \right) (1 + o(1)) .
\end{multline*}
with $o_{h \rightarrow 0}(1)$ uniform in $m$ and dependent on $K$. \\
But $\rho_1 \leq r_{\min} \leq \rho_2$, then for $h,\delta$ sufficiently small we have 
\begin{equation*}
f_{1,h} (m) \geq N \; 2^{2 \valabs{m} -1 } (1+ o(1)). 
\end{equation*}
The function $g(m) = N \; 4^{\valabs{m} -1}$ satisfies the statement conditions.
\end{proof}
\newpage
\appendix
\section{}
 \label{annexe1}
We focus here on the proof of Lemmas \ref{lemme.fct.nulle} and \ref{lemme.theta} in Section \ref{section.invariants}.\\
The first one is a consequence of the Hodge de-Rham theory.
\begin{lemma} 
Let $F \in \cC^\infty \left( \Omega, \RR^2 \right)$ a vector potential satisfying \eqref{eq.potentiel.rot.div.norm} with $B=0$ and
$$
\int_{\partial \Omega_{int}} F = 0,
$$
where $\partial \Omega_{int}$ is defined in \eqref{eq.Omega}.\\
Then, $F=0$.
\end{lemma}
\begin{proof} ~~
\begin{enumerate}
\item Let us show that there exists $\mathbf{G} \in \cC^{\infty} \left( \Omega , \RR \right) $ such that $ F= (F_1 , F_2) = \nabla \mathbf{G} $.\\
Consider the differential form $\omega = F_1 \diff x + F_2 \diff y $. \\
It is closed and has zero integral on $\partial \Omega_{int}$. Indeed,
\begin{equation*}
\diff \omega = \mathrm{rot} (F) \, \diff x \wedge \diff y = 0 ,
\end{equation*}
and 
\begin{equation*}
\int_{\partial A_{int}} \omega = \int_{0}^{2 \pi} F (\gamma(t)) \cdot \gamma'(t) \diff t =  \int_{\partial \Omega_{int}} F = 0,
\end{equation*}
with for all $t \in [0 , 2 \pi [$, $\gamma (t) = ( \rho_1 \cos ( t) , \rho_1 \sin ( t) )$. \\
Thus, according to \cite[Corollaire 9.19 p.130]{fulton2013algebraic}, $\omega$ is exact.
\item Then, since $\mathrm{div} \left( A \right) =0$ and $ A \cdot \mathbf{n} = 0$, we have
\begin{equation*}
\left\| F \right\|^2 = \int_{\Omega} F \cdot \nabla \mathbf{G} = \int_{\partial \Omega} \mathbf{G} \left( A \cdot \mathbf{n} \right) \diff \sigma - \int_{\Omega} \mathrm{div} \left( A \right) \mathbf{G} = 0,
\end{equation*}
with $ \mathbf{n}$ the unit normal to $\partial \Omega$ and $\diff \sigma$ the surfacic measure associated to $\partial \Omega$.
\end{enumerate}
\end{proof}
\begin{lemma} 
Let $\theta$ the unique solution of
\begin{equation*}
      \left\{
      \begin{array}{lllll}
        \Delta \theta = 0 & \text{on} \; \Omega & & & \\
        \theta = 1  & \text{in} \; \partial \Omega_{int} & \text{and} &  \theta  = 0 & \text{in} \; \partial \Omega_{ext}.
     \end{array}
    \right.
\end{equation*}
Then, $\nabla^\perp \theta$ verifies \eqref{eq.potentiel.rot.div.norm} with $B=0$ and we have in polar coordinates, for all $(r,s) \in [\rho_1 , \rho_2 ] \times [0, 2\pi [$,
$$
\nabla^{\perp} \theta (r,s) = \frac{1}{r\; \mathrm{ln}(\rho_1 / \rho_2 )} \begin{pmatrix} -\sin(s) \\ \cos(s) \end{pmatrix} \; .
$$

Moreover
$$
\int_{\partial \Omega_{int}} \nabla^\perp \theta = \frac{2 \pi}{\ln{\left( \rho_1 / \rho_2 \right)}},
$$
with $\partial \Omega_{int}$ defined in \eqref{eq.Omega}.
\end{lemma}
\begin{proof} The existence and uniqueness of a such function follows from the theory of elliptic partial differential equations. \\
In polar coordinates, we have for all $r \in [\rho_1 , \rho_2]$,
$$
\theta (r) = \frac{\mathrm{ln} \left(r / \rho_2 \right)}{\mathrm{ln} \left( \rho_1 / \rho_2  \right)} .
$$
Furthermore, by using
$$
\nabla^\perp = \begin{pmatrix} -\sin(s) \\ \cos(s) \end{pmatrix} \partial_r - \begin{pmatrix} \cos(s) \\ \sin(s) \end{pmatrix} \frac{\partial_s}{r},
$$
it is easy to check that for any $(r,s) \in [\rho_1 , \rho_2 ] \times [0, 2\pi [$,
$$
\nabla^{\perp} \theta (r,s) = \frac{1}{r\; \mathrm{ln}(\rho_1 / \rho_2 )} \begin{pmatrix} -\sin(s) \\ \cos(s) \end{pmatrix} \; .
$$
Then
\begin{equation*}
\int_{\partial \Omega_{int}} \nabla^{\perp} \theta   =   \int_{0}^{2 \pi} \nabla^{\perp} \theta  (\gamma(t)) \cdot \gamma'(t) \diff t =  \frac{2 \pi}{\ln{\left( \rho_1 / \rho_2 \right)}},
\end{equation*}
with for all $t \in [0 , 2 \pi [$, $\gamma (t) = ( \rho_1 \cos ( t) , \rho_1 \sin ( t) )$. 
\end{proof}

            \section{Polar coordinates of the Pauli operator and fibration} \label{annexe_decompo}
In this appendix, we justify the decomposition of the Pauli operator in polar coordinates, given in Section \ref{subsection.fibration_pauli}. We prove then Lemma \ref{lemme.union.sp}.
 
        \subsection{Decomposition in polar coordinates} ~ \\
Pauli matrices 
$$
\sigma_1 =
\begin{pmatrix}
0 & 1 \\
1 & 0
\end{pmatrix},
\quad
\sigma_2 =
\begin{pmatrix}
0 & -i \\
i & 0
\end{pmatrix},
\quad
\sigma_3 =
\begin{pmatrix}
1 & 0 \\
0 & -1
\end{pmatrix}  \, ,
$$
satisfy for all $a,b \in \CC^3$,
\begin{equation} \label{eq.pauli.identite}
    \left(  \sigma \cdot a \right)\left(  \sigma \cdot b \right) = < a,b > I_2 + i \sigma \cdot \left( a \wedge b \right),
\end{equation}
with $\sigma = \left( \sigma_1 , \sigma_2 , \sigma_3 \right)$. 
\begin{remark} \label{remarque.pauli.identite}
A consequence of identity \eqref{eq.pauli.identite} is if we consider
$$
\sigma = \begin{pmatrix}
\sigma_1 \\ \sigma_2
\end{pmatrix} \; , \; \; e_r = \begin{pmatrix}
 \cos(s) \\ \sin (s) 
\end{pmatrix} \; \; \text{and} \; \;   e_s = \begin{pmatrix}
- \sin (s) \\ \cos(s) 
\end{pmatrix} ,
$$
we have
$$
\left( \sigma \cdot e_r \right)^2 = \left( \sigma \cdot e_r \right)^2 = I_2 \; \; \text{and} \; \; \left( \sigma \cdot e_r \right)\left( \sigma \cdot e_s \right) = i \sigma_3 = -\left( \sigma \cdot e_s \right)\left( \sigma \cdot e_r \right) .
$$
\end{remark}

\begin{notation}
We write for all $h\in ]0,1]$ and $p \in \ZZ$,
$$
 \psi = \phi + h \gamma_{h,p} \ln \left( \frac{\rho_1}{\rho_2} \right) \theta ,
$$
with $\gamma_{h,p}$ defined in Proposition \ref{propo.potentiel.flux}. \\
Thus, we have $\mathbf{A_{h,p}} = \nabla^\perp \psi$.
\end{notation}

\begin{lemma}
The Dirichlet-Pauli operator in polar coordinates, denoted by $\widetilde{\cP_h}$, acts on $L^2 \autopar{ ]\rho_1, \rho_2 [ \times [0,2\pi[,\CC^2; r \diff r \diff s}$ as
\begin{equation*} 
\widetilde{\cP_h} = \lc -h^2 \autopar{\partial_{rr}^2 \; + \; \frac{1}{r} \partial_r } \; + \; \autopar{ h \;  \frac{\autopar{-i\partial_s - \gamma_{h,p} } }{r} - \partial_r \phi(r) }^2 \rc I_2 \; - \; hB(r) \sigma_3 ,
\end{equation*}
with $\gamma_{h,p}$ defined in Proposition \ref{propo.potentiel.flux}. \\
\end{lemma}

\begin{proof} ~ \\
Recall that for $h > 0$ and $p \in \ZZ$, we have
\begin{equation*}
    \cP_h = \left[  \sigma  \cdot \left( -ih\nabla - \mathbf{A_{h,p}} \right) \right]^2 .
\end{equation*}
By writing the gradient in polar coordinates, we get
$$
\sigma  \cdot \left( -ih\nabla - \mathbf{A_{h,p}} \right) = \left(  \sigma \cdot e_r \right) \lb -ih \partial_r \rb + \left(  \sigma \cdot e_s \right) \lb -\frac{ih}{r} \partial_s - \partial_r \psi \rb 
$$
with $e_r = ( \cos(s) , \sin (s) )$ and $e_s = (- \sin (s) , \cos(s) )$. \\
We are reduced to compute the square of an operator of type $(R + S + T)$ with
$$
R = \left(  \sigma \cdot e_r \right) \lb -ih \partial_r \rb \; , \; \; S = \left(  \sigma \cdot e_s \right) \lb -\frac{ih}{r} \partial_s \rb \; \; \text{and} \; \; T = \left(  \sigma \cdot e_s \right) \lb - \partial_r \psi \rb . 
$$
Let $u \in \cC^{\infty} \left(]\rho_1, \rho_2 [ \times [0,2\pi[,\CC^2 \right)$, by using Remarque \ref{remarque.pauli.identite}, we have
\begin{align*}
    A(A+B+C) u & = I_2 \lb - h^2 \partial_{rr}^2 \rb u + i \sigma_3 \lb \frac{h^2}{r^2} \partial_s - \frac{h^2}{r} \partial^2_{sr} + ih \partial_r \psi \partial_r + ih \partial^2_{rr} \psi \rb u \\
    C(A+B+C) u &= -i\sigma_3 \lb ih \partial_r \psi \partial_{r} \rb u + I_2 \lb \frac{ih}{r} \partial_r \psi \partial_s + \left( \partial_r \psi \right)^2 \rb u .
\end{align*} 
Then, we get
\begin{align*}
    BA u &= \left( \sigma \cdot e_s \right) \lb \left( \sigma \cdot e_s \right) \lc  - \frac{h^2}{r} \partial_r \rc u + \left( \sigma \cdot e_r \right) \lc  - \frac{h^2}{r} \partial^2_{sr} \rc\rb u \\
    & = I_2 \lb - \frac{h^2}{r} \partial_r \rb u + i \sigma_3 \lb \frac{h^2}{r} \partial^2_{sr} \rb u ,
\end{align*}
and 
\begin{align*}
    B(B+C) u &= \left( \sigma \cdot e_s \right) \lb \left( \sigma \cdot e_r \right) \lc \frac{h^2}{r^2} \partial_s - \frac{ih}{r} \partial_r \psi \rc u + \left( \sigma \cdot e_s \right) \lc - \frac{h^2}{r^2} \partial^2_{ss} + \frac{ih}{r} \partial_r \psi \partial_s \rc u\rb \\
    & = i \sigma_3 \lb \frac{ih}{r} \partial_r \psi - \frac{h^2}{r^2} \partial_s  \rb u + I_2 \lb - \frac{h^2}{r^2} \partial^2_{ss} + \frac{ih}{r} \partial_r \psi \partial_s \rb u .
\end{align*}
By combining the terms, we get
\begin{multline*}
    \widetilde{\cP_h} u =  \left( -h^2 \lc \partial^2_{rr} + \frac{1}{r} \partial_r  \rc + \lc \left( \partial_r \psi \right)^2 + 2 \left( \partial_r \psi \right)  \frac{ih \partial_s}{r} - \frac{h^2}{r^2} \partial^2_{ss} \rc  \right) I^2 u \\ 
    -h \left( \partial^2_{rr} \psi + \frac{1}{r} \partial_r \psi  \right) \sigma_3 u .
\end{multline*}
By Lemma \ref{lemme.theta}, we have
\begin{equation} \label{eq.annexe.partial_rpsi}
    \partial_r \psi = \partial_r \phi + \frac{h \gamma_{h,p}}{r} .
\end{equation}
Finally, by factorizing and applying equation \eqref{eq.annexe.partial_rpsi}, we get
\begin{equation*} 
\widetilde{\cP_h} = \lc -h^2 \autopar{\partial_{rr}^2 \; + \; \frac{1}{r} \partial_r } \; + \; \autopar{ h \;  \frac{\autopar{-i\partial_s - \gamma_{h,p} } }{r} - \partial_r \phi(r) }^2 \rc I_2 \; - \; hB(r) \sigma_3 .
\end{equation*}
\end{proof}

\subsection{Spectrum of fibered operator} ~ \\
We prove here Lemma \ref{lemme.union.sp}.
\begin{proof} Recall that $\cP_h$ acts as $\widetilde{\cP_h }$ in polar coordinates and $\widehat{\cP_h } = r^{1/2} \widetilde{\cP_h } r^{-1/2}$. We thus have the identity
$$
Sp \autopar{ \cP_h } = Sp \autopar{ \widetilde{\cP_h } } = Sp \autopar{\widehat{\cP_h }}.
$$
Let us show the second equality.
\begin{itemize}
\item Let $\lambda \in \displaystyle{\bigcup_{m\in \ZZ} Sp \autopar{\cP_{h,m}}}$, there exist $m \in \ZZ$ and $v_m \in  H_{0}^1  \cap H^2 \autopar{ [\rho_1, \rho_2 ],\CC^2}$ such that $\cP_{h,m} v_m = \lambda v_m$.\\
We have
$$
\lambda u_m (r) e^{im s} = \cP_{h,m} \autopar{u_m (r)} e^{ims} = \widehat{\cP_{h}} \autopar{v_m (r) e^{ims}}.
$$
This shows the first inclusion.
\item Let $\lambda \in Sp \autopar{ \widehat{\cP_h} } $, $\ie$ there exists $v \in H_{0}^1  \cap H^2 \autopar{[\rho_1, \rho_2 ]\times [0,2\pi[,\CC^2; \diff r \diff s} \backslash \{ 0 \}$ such that $ \widehat{\cP_h}  v = \lambda v$.\\
We can decompose the eigenfunction into Fourier series
$$
v(r,s) = \sum_{m\in \ZZ} v_m (r) e^{im s}.
$$
Then, in the sense of tempered distributions, we have
\begin{equation*}
  0 = \left(  \widehat{\cP_h}  - \lambda \Id \right) v(r,s) = \sum_{m\in \ZZ}  \autopar{ \cP_{h,m} - \lambda \Id} v_m (r) e^{im s} 
\end{equation*}
But $v\neq 0$, so there is at least one $m\in \ZZ$ such that $u_m \neq 0$, $\ie$
$$
\cP_{h,m} u_m = \lambda u_m .
$$
\end{itemize}
\end{proof}

\subsection*{Acknowledgements} I am very grateful to Nicolas Raymond and Loïc Le Treust for our many discussions and their valuable support throughout this work. This work is supported by ANR DYRAQ ANR-17-CE40-0016-01 and IRP - CNRS SPEDO.

\end{document}